\documentclass{amsart}
\usepackage[a4paper,top=2cm,bottom=3cm,left=4cm,right=4cm]{geometry}
\usepackage[T1]{fontenc}
\usepackage{newlfont}
\usepackage[utf8]{inputenc}
\usepackage[english]{babel}
\usepackage{graphicx}
\usepackage{amsmath}
\usepackage{amsthm}
\usepackage{wrapfig}
\usepackage{latexsym}
\usepackage{amssymb,amsfonts}
\usepackage{tikz}
\usepackage{subfigure}
\usepackage{xcolor}
\usepackage{verbatim}
\usepackage[normalem]{ulem}

\graphicspath{{Figures/}}
\usepackage[pdfencoding=auto, psdextra]{hyperref}

\newcommand{\Z}{{\mathbb Z}}
\newcommand{\D}{{\widetilde{D}_{n+2}}}
\newcommand{\B}{{\widetilde{B}_{n+1}}}
\newcommand{\Di}{{\mathbb D}}
\newcommand{\ld}{{\mathrm{D}_{\mathrm{L}}}}
\newcommand{\rd}{{\mathrm{D}_{\mathrm{R}}}}
\newcommand{\tl}{{\mathrm{TL}}}
\newcommand{\fc}{{\mathrm{FC}}}

\newcommand{\bO}{{\mathcal{L}_{\bullet}}}
\newcommand{\wO}{{\mathcal{L}_{\circ}}}
\newcommand{\bwO}{{\mathcal{L}_{\bullet}^{\circ}}}

\newcommand{\s}{{s_{\bullet}}}
\newcommand{\ti}{{t_{\bullet}}}
\newcommand{\sn}{{s_\circ}}
\newcommand{\tn}{{t_\circ}}
\newcommand{\TLR}{T^{LR}_k(\Omega)}
\newcommand{\PLR}{{\mathcal{P}^{LR}_k(\Omega)}}
\newcommand{\PLRk}{\widehat{\mathcal{P}}^{LR}_{k}(\Omega)}
\newcommand{\PLRn}{\widehat{\mathcal{P}}^{LR}_{n+2}(\Omega)}
\newcommand{\su}{\mathrm{Suff}}
\newcommand{\pr}{\mathrm{Pref}}
\newcommand{\TTC}{\mathrm{CC}}
\newcommand{\TTCw}{\mathrm{CC_w}}
\newcommand{\zz}{\mathrm{CZ}}
\newcommand{\car}{\mathrm{K}}
\newcommand{\carw}{\mathrm{K_w}}
\newcommand{\IR}{\mathrm{I_w}}
\newcommand{\IRc}{\mathrm{I_w'}}
\newcommand{\I}{\mathrm{I}}
\newcommand{\nmap}{\varphi}

\newcommand{\fzp}{A}
\newcommand{\fzs}{B}
\newcommand{\red}{\rightsquigarrow}
\newcommand{\Zd}{\mathbb{Z}[\delta]}
\newcommand{\ab}{\mathbf{a}}
\newcommand{\at}{\widetilde{\mathbf{a}}}

\newtheorem{Theorem}{Theorem}[section]
\newtheorem{Lemma}[Theorem]{Lemma}
\newtheorem{Proposition}[Theorem]{Proposition}
\newtheorem{Corollary}[Theorem]{Corollary}

\theoremstyle{definition}
\newtheorem{Definition}[Theorem]{Definition}
\newtheorem{Remark}[Theorem]{Remark}

\newtheorem{Example}[Theorem]{Example}
\newtheorem{Notation}[Theorem]{Notation}

\title[Star and weak star irreducible FC elements]{Star and weak star irreducible fully commutative elements in Coxeter groups of affine types $\widetilde{B}$ and $\widetilde{D}$}

\author[R. Biagioli]{Riccardo~Biagioli}
\author[L. Costantini]{Luca~Costantini}
\author[E. Sasso]{Elisa Sasso}
\address{Riccardo Biagioli, Luca Costantini, Elisa Sasso: Dipartimento di Matematica, Università di Bologna\\ Piazza di Porta San Donato 5, 40126 Bologna, Italy}
\email{riccardo.biagioli2@unibo.it, luca.costantini5@studio.unibo.it, elisa.sasso2@unibo.it}

\begin{document}

\begin{abstract}

The star operation, originally introduced by Kazhdan and Lusztig, was later specialized by Ernst to the so-called weak star reduction on the set of fully commutative elements of a Coxeter group. In this paper, we classify the star and weak star irreducible fully commutative elements in Coxeter groups of affine types $\B$ and $\D$. Focusing then on the case of type $\D$, we use the classification of star irreducible elements to provide a new proof of the faithfulness of a diagrammatic representation of the corresponding generalized Temperley–Lieb algebra, along with an explicit description of Lusztig’s 
$\mathbf{a}$-function.

\end{abstract}

\maketitle

%%%%%%%%%%%%%%%%%%%%%%%%%%%%%%
\section*{Introduction}
%%%%%%%%%%%%%%%%%%%%%%%%%%%%%%

The notion of the star operation was first introduced by Kazhdan and Lusztig in \cite{kazhdan_reps} for simply laced Coxeter systems and later extended to arbitrary Coxeter systems in \cite{lusztig_cells}. Star operations provide a powerful tool to study the decomposition of a Coxeter group into Kazhdan–Lusztig cells. These cells play a fundamental role in representation theory. 
\smallskip

In this paper, we consider a special case of the star operation that strictly decreases the length of the element to which it is applied \cite{GreenStar}, and  following \cite{ernst}, we focus on the star and weak star reducibility of fully commutative elements. Over the years, star irreducible elements have been classified. Green \cite{GreenStar} defined a Coxeter group to be \textit{star reducible} if and only if all of its star irreducible elements are products of commuting generators. In particular, he proved \cite[Theorem 6.3]{GreenStar} that finite types $A,B,D,E,F,H,I$ and affine types $\widetilde{A}_n$ (with $n$ even), $\widetilde{C}_n$ (with $n$ odd), $\widetilde{E}_6$, and $\widetilde{F}_5$ are star reducible Coxeter systems. 
In \cite{ernst}, Ernst relaxed this notion by introducing  \textit{weak star reducibility}  (also referred to as \textit{cancellability}, following Fan \cite{Fan}), which coincides with star reducibility in the simply laced cases. He classified the weak star irreducible (or non-cancellable) elements of Coxeter systems of types $B$ and $\widetilde{C}$ in \cite[Theorem 4.2.1, Theorem 5.1.1]{ernst}. 
\smallskip

In this paper, we classify the star irreducible and weak star irreducible elements for the affine types $\D$ and $\B$. Beyond their intrinsic interest as tools in the representation theory of Hecke algebras, these (weak) star irreducible elements enable us to give a new proof of the faithfulness of the diagrammatic representation of the Temperley–Lieb algebra of type $\D$ defined by the first and third authors in \cite{BFS}. While this result was originally established using topological and combinatorial arguments based on diagram manipulations, our proof, motivated by valuable suggestions from the anonymous referee of \cite{BFS}, takes a more algebraic approach that relies crucially on star irreducibility. Finally, we present a method to compute Lusztig’s $\ab$-function \cite{lusztig_cells} for fully commutative elements in terms of their associated decorated diagrams.
\smallskip

This paper is organized as follows.  Section \ref{se1} provides a brief review of essential facts about Coxeter groups and the \textit{Cartier--Foata normal form}, which will be used extensively throughout the paper. We also introduce a particular partial order, known as \textit{heap} of an element, whose Hasse diagram helps to better visualize the structure of (weak) star irreducible elements.
In Section \ref{se2}, we classify the star irreducible elements of type $\D$ (Theorem \ref{decimo}). 
In Section \ref{se3}, building on this classification and employing an injective map between the sets of fully commutative elements of types $\D$ and $\B$ (Definition \ref{def:mu}), we also characterize the star irreducible elements (Theorem \ref{theorem:classredB}) and weak star irreducible elements (Theorem \ref{undicesimo}) of type $\B$. In Sections \ref{sec:diagrams} and \ref{sec:descents}, we recall the definition of a family of decorated diagrams, introduced in \cite{BFS}, and study their properties in relation to the irreducible elements of type $\D$. 
In Section \ref{sec:injectivity}, we provide an alternative proof of the faithfulness of the diagrammatic representation of $\tl(\D)$ (Theorem \ref{lastresult}) arising from such decorated diagrams. In the final section, we give a characterization of Lusztig's $\ab$-function in terms of parameters on the decorated diagrams of type $\D$ (Theorem \ref{theorem:a-lusztig}).

%%%%%%%%%%%%%%%%%%%%%%%%%%%%%%%%%%%%%%
\section{Preliminaries} \label{se1}
%%%%%%%%%%%%%%%%%%%%%%%%%%%%%%%%%%%%%%

\subsection{Coxeter groups and fully commutative elements}
Let $M$ be a square symmetric matrix indexed by a finite set $S$, satisfying  $m_{s,t}=m_{t,s}\in \mathbb{N}_{>0}\cup \{\infty\}$ and $m_{s,t}=1$ if and only if $s=t$, for all $s,t\in S$. The \textit{Coxeter group} $W$ associated with the \textit{Coxeter matrix} $M$ is defined by generators $S$ and relations $(st)^{m_{s,t}}=e$ if $m_{s,t}< \infty$. These relations can be rewritten as $s^2=e$ for all $s\in S$, and \begin{equation}\label{eq:mst}    
[st]_{m_{s,t}}:=\underbrace{sts\cdots}_{m_{s,t}}=\underbrace{tst\cdots}_{m_{s,t}}=[ts]_{m_{s,t}}\end{equation} for all $s,t\in S$ such that $2\le m_{s,t}< \infty$, the latter being called \textit{braid relations}. When $m_{s,t}=2$, they are simply \textit{commutation relations} $st=ts$. 

The \textit{Coxeter graph} $\Gamma$ associated with a Coxeter system $(W, S)$ is the labeled graph whose vertex set is $S$, with an edge between distinct elements $s, t \in S$ whenever $m_{s,t} \ge 3$. The edge is labeled with $m_{s,t}$ if $m_{s,t} \ge 4$. Therefore non adjacent vertices correspond precisely to commuting generators. We say that $(W,S)$ is \emph{irreducible}, when the associated Coxeter graph $\Gamma$ is connected. 

\begin{figure}[h!]
    \centering
    \includegraphics[width=0.7\linewidth]{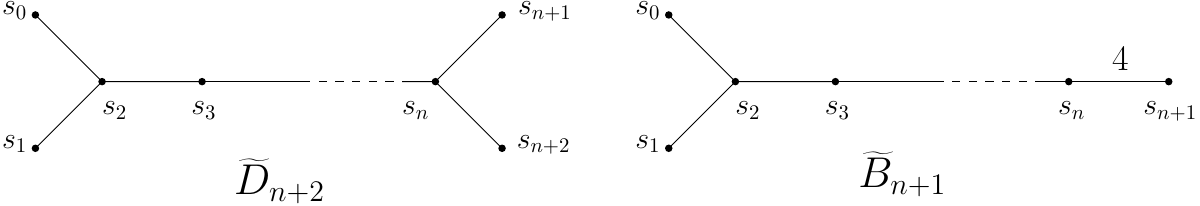}
    \caption{Coxeter graphs of type $\D$ and $\B$.}
    \label{coxgra}
\end{figure}
  
In this section, we fix $(W,S)$ to be a Coxeter system. The \textit{length} function is $\ell:W\rightarrow \mathbb{N}$, is defined as follows: $\ell(e)=0$ and for every $e\ne w\in W$
\[\ell(w):=min\{r\in \mathbb{N}\mid w=s_{i_i}\cdots s_{i_{r}}\mbox{, }s_{i_k}\in S\}.\]

An expression of $w \in W$ of minimal length is called \textit{reduced}. We denote by $\mathcal{R}(w)$ the set of all reduced expressions of $w$. We use boldface to indicate a particular reduced expression.

We define the \textit{left descents set} and the \textit{right descents set} of $w\in W$ respectively by 
\[\ld(w):=\{s\in S\mid \ell(sw)<\ell(w)\} \mbox{ and }\rd(w):=\{s\in S\mid \ell(ws)<\ell(w)\}.\]

We have that $s\in \ld(w)$ ($s\in \rd(w)$) if and only if exists a reduced expression for $w$ such that starts with $s$ (finishes with $s$). For more details see \cite[\S 1.4]{b07}.

Let $u, v, w \in W$. We say that $w = uv$ is a \textit{reduced product}  if $\ell(w) = \ell(u) + \ell(v)$. In this case, $u$ is called a  \textit{prefix} of $w$, and $v$ is called a \textit{suffix} of $w$. In general, $u$ is a \textit{factor} of $w$ if there exist $x, y \in W$ such that $w = xuy$ and $\ell(w) = \ell(x) + \ell(u) + \ell(y)$.
 We denote by $\pr(w)$ (respectively, $\su(w)$) the set of prefixes (respectively, suffixes)  of $w$.

A fundamental result in Coxeter group theory, sometimes called the \textit{Matsumoto Theorem}, states that any reduced expression of $w$ can be obtained from any other reduced expression of $w$ by using only braid relations. 
\begin{Definition}
\label{fcdefinition}
 We say that $w\in W$ is \textit{fully commutative} (FC) if any reduced expression of $w$ can be obtained from any other reduced expression of $w$ by using only commutation relations.
The set of fully commutative elements in $W$ is denoted by $\fc(W)$.
\end{Definition}

The next proposition, due to Stembridge \cite[Proposition 2.1]{stem}, characterizes the FC elements and it is useful to test whether a given element is FC or not.

\begin{Proposition}
\label{caratterizzazionefc}
Let $w\in W$. Then, $w\in \fc(W)$ if and only if for all $s,t\in S$ such that $3\le m_{s,t}< \infty$, there is no reduced expression of $w$ having $[st]_{m_{s,t}}$ as a factor.
\end{Proposition}

Let $\mathbf{w}=s_{i_1}\cdots s_{i_r}$ a reduced expression of $w\in W$.
Set $[r]:=\{1,2,\ldots, r\}$; we define the \textit{support} of $w\in W$ the set \[supp(w):=\{s\in S\mid \exists k\in [r]\mbox{, }s=s_{{i_k}}\}.\] 
Note that by Matsumoto Theorem, this set does not depend on the choice of $\mathbf{w}$. In particular, if $w\in\fc(W)$ and $\mathbf{w'}=s_{j_1}\cdots s_{j_r}$ is another reduced expression of $w$, it follows that
\[\left | \{k\in \left[r\right]\mid s_{i_k}=s\}\right |=\left |\{h\in \left[r\right ]\mid s_{j_h}=s\}\right |,\]
for all $s\in supp(w)$.\\

The following subset of fully commutative elements plays an important role in this paper.

\begin{Definition}
\label{comm}
We say that $w\in \fc(W)$ is \textit{completely commutative} if it is a product of commuting generators. We denote this subset by $\TTC(W)$ and we assume that $e\in \TTC(W)$. 
\end{Definition}

We now introduce two important closely related tools for studying $\fc$ elements: the Cartier–Foata normal form and the heap of an element, introduced in \cite{Fan} and \cite{viennot}, respectively. 

\begin{Theorem}[Cartier--Foata normal form]
\label{normalform}
Every $w\in \fc (W)$ admits a unique factorization   
\begin{equation} \label{CFNF}
w=u_0 \cdots u_m   
\end{equation}
called the \textit{Cartier-Foata normal form} (CFNF), satisfying the following conditions: 
\begin{enumerate}
    \item[(a)] $supp(u_0)=\ld(w)$; 
    \item[(b)] $u_i\in \TTC(W)$ for all $i\in \{0,1,\ldots, m\}$;
    \item[(c)] $\ell(w)=\ell(u_0)+\cdots +\ell(u_m)$;
    \item[(d)] if $t\in supp(u_{j+1})$, $j\in \{0,1\ldots, m-1\}$, then there exists $s\in  supp(u_j)$ such that $m_{s,t}\ge 3$.
\end{enumerate}
\end{Theorem}

\begin{Lemma}
\label{lemmaunitarissimo}
Let $w\in \fc(W)$ and $w=u_0 \cdots u_m$ be its CFNF. Let $z=s_{i_0}\cdots s_{i_k}$ be a reduced expression of a prefix of $w$ satisfying the following conditions: 
\begin{itemize}
    \item[(a)] if $k\ge1$, $s_{i_j}$ and $s_{i_{j+1}}$ are not commuting generators for all $0\le j\le k-1$;
    \item[(b)] $s_{i_k}\in supp(u_m)$;
    \item[(c)] $w=zv$ is a reduced product and $s_{i_k}\notin supp(v)$.
\end{itemize}
Then $k=m$. 
\end{Lemma}

\begin{proof}
We have that $s_{i_0}\in supp(u_0)$, since $s_{i_0}\in \ld(w)$. Since $s_{i_0}\cdots s_{i_j}$ is prefix of $w$ for each $1\le j\le k$, by (a) and Theorem \ref{normalform}(d), we have that $s_{i_j}\in supp(u_j)$ and $z$ is a prefix of $u_0\cdots u_k$. Furthermore, by (b) and (c) we can conclude that $k=m$. Otherwise if $k<m$, then $w=zv$ is a reduced product and $supp(u_m)\subseteq supp(v)$, so $s_{i_k}\in supp(v)$, that is a contradiction. 
\end{proof}

\begin{Definition}
\label{heap}
Let $w\in W$. Consider $\mathbf{w}=s_{i_1}\cdots s_{i_r}\in \mathcal{R}(w)$, we define a partial ordering $\prec$ on $[r]$ via the transitive closure of the relation \[j\prec k \mbox{ if }j<k \mbox{ and } m_{s_{i_j},s_{i_k}}\ge 3.\]
In particular, $j\prec k$ if $j<k$ and $s_{i_j}=s_{i_k}$. The \textit{heap} of $\mathbf w$ is the labeled poset given by the triple $H({\mathbf w}):=([r], \prec, \varepsilon)$, where $ \varepsilon:j\mapsto s_{i_j}$ is the labeling map.
\end{Definition}

It is well-known that if $w\in \fc(W)$ and $\mathbf{w},\mathbf{w'}\in \mathcal{R}(w)$, then $H(\mathbf{w})\cong H(\mathbf{w'})$ as labeled poset. Hence, it is well defined $H(w):=H(\mathbf{w})$, where $\mathbf{w}\in \mathcal{R}(w)$, see \cite{viennot}.

\begin{Remark}\label{not:label}
Although the proofs presented in this paper do not rely on the concept of heap, visualizing the heap of an element can provide valuable intuition about its structure. This is particularly helpful for the star-irreducible elements discussed later, whose names are inspired by the shapes of their associated heaps. We follow the notation of \cite{BFS} to depict the Hasse diagram of a heap, which we now briefly recall. In the Hasse diagram of $H(w)$, $w\in \fc(\D)$, vertices with the same labels are drawn in the same column. In particular, we represent heaps in $n+1$ columns with the following criterion. Both vertices with label $s_0$ and $s_1$ (respectively $s_{n+1}$ and $s_{n+2}$) will be drawn in the first (respectively last) column of $H(w)$. Moreover, when $s_0s_1$ (respectively $s_{n+1}s_{n+2}$) is a factor of $w$, we represent the vertices associated to these $s_0$ and $s_1$ (respectively $s_{n+1}$ and $s_{n+2}$) with a single mark point labeled $s_0s_1$ (respectively $s_{n+1}s_{n+2}$). In a similar way, we represent a heap of $w\in \fc(\B)$ with $n+1$ columns, where the criterion for the placement of double vertices is applied only to $s_0$ and $s_1$. 

Additionally, we represent the Hasse diagram of  $H(w)$ so that all entries corresponding to elements in
$\ld(w)$ appear at the same vertical level in the topmost row of the diagram, while all other entries are placed as high as possible, subject to the partial order. Given this representation of $H(w)$, the CFNF of $w$ can be obtained by reading the labels of the vertices row by row, from top to bottom, and from left to right within each row. 
\end{Remark}

\begin{Example}
\label{exheap}
Consider the following elements written in their CFNF: 
\begin{align*}
w_1&=(s_0s_4)(s_3s_5)(s_2s_4s_6s_7)s_1\in \fc(\widetilde{D}_{7});\\
w_2&=(s_3)(s_2s_4) (s_1s_3s_5)(s_2s_4s_6)(s_0s_3s_5)(s_2s_6)\in \fc(\widetilde{B}_{6}).
\end{align*}
 Their heaps are depicted in Figure $\ref{exampleheap}$. 

\begin{figure}[ht]
    \centering
    \includegraphics[width=0.5\linewidth]{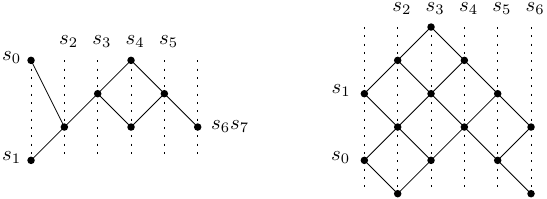}
    \caption{Hasse diagrams of $H(w_1)$ and $H(w_2)$ of Example \ref{exheap}.}
    \label{exampleheap}
\end{figure}
\end{Example}

Fully commutative elements in a Coxeter group $W$ index a basis of the so-called generalized Temperley--Lieb algebra of $W$. The following is the presentation of $\mathrm{TL}(\D)$ given by Green in \cite[Proposition 2.6]{GreenStar} that will be used in our work as a definition. The Coxeter graph of type $\D$ is depicted in Figure \ref{coxgra}.

\begin{Definition}\label{def:tl-algebras}
The \textit{Temperley--Lieb algebra of type $\D$}, denoted by $\mathrm{TL}(\D)$, is the $\Zd$-algebra generated by $\{b_0, b_1, \ldots, b_{n+2}\}$ with defining relations:

\begin{enumerate}
\item[(d1)] $b_i^2=\delta b_i$ for all $i\in \{0,\ldots, n+2\}$;
\item[(d2)] $b_i b_j=b_jb_i$ if $s_i$ and $s_j$ are not adjacent nodes in the Coxeter graph;
\item[(d3)] $b_i b_j b_i=b_i$ if $s_i$ and $s_j$ are adjacent nodes in the Coxeter graph.
\end{enumerate}
\end{Definition}

For any $s_{i_1}\cdots s_{i_k}$ reduced expression of $w$ in $\fc(\D)$, we set
$$b_w=b_{i_1} \cdots \ b_{i_k}.$$
It is easy to see that $b_w$ does not depend on the chosen reduced expression of $w$. In \cite{Graham}, Graham proved that the set $\{b_w \mid w\in \fc(\D)\}$ is a basis for $\tl(\D)$, which is usually called the \textit{monomial basis} (see also \cite[Proposition 2.4]{GreenStar}).

\subsection{Star and weak star reducible $\fc$ elements}

In this section we fix an arbitrary irreducible Coxeter system $(W,S)$.
Now, we recall the notions of star and weak star reducibility in $(W,S)$ and we refer the reader to \cite{ernst, GreenStar, green-trace} for more details. 
\begin{Definition}
\label{def:star-red}
    Let $w\in \fc(W)$ and suppose that $s\in \ld(w)$ (respectively, $s\in \rd(w)$). We say that $w$ is \emph{left (respectively, right) star reducible by $s$ with respect to $t$} to $sw$ (respectively, $ws$) if there exists $t\in \ld(sw)$ (respectively, $t\in \rd(ws)$) with $m_{s,t}\ge 3$. 
\end{Definition}

Observe that if $m_{s,t}\ge 3$, then $w$ is left (respectively, right) star reducible by $s$ with respect to $t$ if and only if $w=stv$ (respectively, $w=vts$), with $\ell(w)=\ell(v)+2$.

We now define the concept of weak star reducible elements, first introduced in \cite{ernst}, which is related the notion of cancellable elements introduced by Fan in \cite{Fan}. 

\begin{Definition}
\label{irr}
Let $w \in \fc(W)$ and $s,t\in S$, we say that $w$ is \textit{left (respectively, right) weak star reducible by $s$ with respect to $t$ to $sw$ (respectively, $ws$)}, if the following hold:
\begin{enumerate}
    \item[(a)] $w$ is left (respectively, right) star reducible by $s$ with respect to $t$ to $sw$ (respectively, $ws$);
    \item[(b)] $tw\notin \fc(W)$ (respectively,  $wt\notin \fc(W)$). 
\end{enumerate}

\end{Definition}

Let $w\in \fc(W)$, we say that $w$ is \emph{left (respectively, right) (weak) star irreducible} if does not exist $s,t\in S$ such that $w$ is left (respectively, right) (weak) star reducible by $s$ with respect to $t$. Moreover, we say that $w$ is \emph{(weak) star irreducible} if it is both left and right (weak) star irreducible. We set also 
\begin{align*}
    \I(W)&=\{w\in \fc(W) \mid w \mbox{ is star irreducible}\}, \\
    \IR(W)&=\{w\in \fc(W) \mid w \mbox{ is weak star irreducible}\}.
\end{align*}
Clearly, $\I(W)\subseteq \IR(W)$ and denote by $\IRc(W):=\IR(W)\setminus\I(W)$. Observe that when $W$ is simply laced (i.e., $m_{s,t}\leq 3$ for all $s,t\in S)$ the definitions of star reducible and weak star reducible are equivalent, therefore we will distinguish these two definitions only in Section \ref{se3} when dealing with type $\B$.  
\smallskip

From now on, to simplify the writing, we will say \textit{reducible} (\textit{irreducible}) instead of \textit{star reducible} (\textit{star irreducible}). Moreover, we use the notation $$w\red_t sw \quad\mbox{(respectively, } w\red_t ws)$$ to denote that $w$ is (weak) left (respectively, right) reducible by $s$ with respect to $t$; we omit the $t$ when we it is not necessary. In general, by $w\red v$ we mean that either $w\red_t sw$ or $w\red_t ws$ for some $s,t\in S$ with $m_{s,t}\ge 3$. We will explicitly write ``weak'' when we consider a weak star reduction.

\begin{Remark}\
\label{oss}

\begin{enumerate}
    \item[(1)] If $w\red_t sw$ (respectively, $w\red_t ws$), then $\ell(tw)>\ell(w)$ (respectively, $\ell(wt)>\ell(w)$).%, since $s\in \ld(w)$ (respectively, $s\in \rd(w)$), $m_{s,t}\ge 3$ and $w\in \fc(W)$.
    \item[(2)] We have that $w$ is left (weak) irreducible if and only if $w^{-1}$ is right (weak) irreducible. Thus, $w$ is (weak) irreducible if and only if $w^{-1}$ is also (weak) irreducible.
    \item[(3)] If $3\le m_{s,t}<\infty$, define $\xi_{s,t} := [st]_{m_{s,t}-1}$ (see \eqref{eq:mst}). We have that $w\red_t sw$ (respectively, $w\red_t ws$) is a weak reduction if and only if $\xi_{s,t}$ is a prefix (respectively, $\xi_{s,t}^{-1}$ is a suffix) of $w$.    
   % and consider the pair $I:=\{s,t\}$. We define $\xi_{s,t} \in W_I$ such that $\ell(\xi_{s,t})=m_{s,t}-1$ and $\ld(\xi_{s,t})=\{s\}$. We have that $w$ is left (respectively, right) weak star reducible by $s$ w.r.t. $t$ if and only if $\xi_{s,t}$ is a prefix (respectively, $\xi_{s,t}^{-1}$ is a suffix) of $w$. ??
    \item[(4)] Assume $m_{s,t}=3$ and $w=u_0\cdots u_m$ be its CFNF, $w\red_t sw$ if and only if $s\in supp(u_0)$, $t\in supp(u_1)$ and for every $s'\in S$ such that $m_{s',t}= 3$, $s'\notin supp(u_0)$. 
    \item[(5)] Assume $s\in \ld(w)$ (resp. $s\in \rd(w)$). If $t\in S$ such that $m_{s,t}\ge 3$ and $tw\notin \fc(W)$ (respectively, $wt\notin \fc(W)$), then $w$ admits $st$ as a prefix (respectively, $ts$ as a suffix), so $w\red_t sw$ (respectively, $w\red_t ws$) is a weak reduction. 
    \item[(6)] Assume $s\ne s'$, if $w\red_t sw$ and $w\red_{t'} s'w$ (respectively, $w\red_t ws$ and $w\red_{t'} ws'$) are (weak) reductions, then $m_{x,x'}=2$ for all $x\in \{s,t\}$ and $x'\in \{s',t'\}$, therefore $sw\red_{t'} s'sw$ (respectively, $ws \red_{t'} wss'$) is a (weak) reduction. 
    \item[(7)] % w=u's' con t'\in \rd(u') e st prefisso di u', quindi t\in \ld(su') e t'\in \rd(su') allora su's' ha come prefisso t's'.  
    Let $w\red_t sw$ and $w\red_{t'} ws'$ be (weak) reductions. Then $ws'\red_t sws'$  is a (weak) reduction if and only if $sw \red_{t'} sws'$ is a (weak) reduction. 
\end{enumerate}
\end{Remark}

\begin{Definition}
\label{starred}
Let $w,v\in \fc(W)$. We say that $w$ is \emph{(weak) reducible to $v$} if there is a sequence (also trivial) of (weak) reductions
\begin{equation} \label{eq:seq}
    w_0=w\red w_1\red w_2\red \cdots \red w_l=v.
\end{equation}
%where for each $0\le i\le l-1$, $w_i$ is left or right (weak) reducible by $s_i$ w.r.t $t_i$ to $w_{i+1}$. %In particular, we call $w_i \mapsto w_{i+1}=s_iw_i$ (respectively, $w_i\mapsto w_{i+1}=w_is_i$) a \emph{left (respectively, right) reduction}.
\end{Definition}

In every step of \eqref{eq:seq}, the length of the elements decreases by 1, therefore, any $w\in \fc(W)$ can be reduced to a (weak) irreducible element.

The following two results are probably known, but we include their proofs due to the lack of precise references.

\begin{Lemma}
\label{lemma:lemmaseq}
   Let $w\in \fc(W)$ be (weak) reducible to $v$ (weak) irreducible with a sequence
   \begin{equation} \label{eq:primseq}
       w_0=w\red w_1\red w_2\red \cdots \red w_l=v.
   \end{equation}
   If there exists a (weak) reduction $w\red w'_1$, with $w_1'\ne w_1$, then there exists $v'$ (weak) irreducible such that $w$ is (weak) reducible to $v'$ by a sequence  \[w_0'=w\red w_1'\red w_2'\red \cdots \red w_l'=v'.\] Moreover, if $|S|>2$ and $supp(v)$ is complete, then $v=v'$.
\end{Lemma} 

\begin{proof}
    We first prove the statement for star reducibility. 
    We assume without loss of generality that $w\red_tsw=:w_1'$, with $w_1'\ne w_1$. By hypothesis, $z:=st$ is a prefix of $w$, hence in sequence \eqref{eq:primseq}, there exists $1\le j\le l-1$ such that $z$ is a prefix of $w_i$ for all $1\le i\le j$ but $z$ is not a prefix of $w_{j+1}$. Then, either $w_j\red_t sw_j=w_{j+1}$, or $w_j\red_s w_jt=w_{j+1}$.   
    Moreover, since $w_i \red w_{i+1}$, by Remark \ref{oss}(6) and (7), $sw_i$ is reducible to $sw_{i+1}$ for all $0\le i\le j-1$.
    Now, we distinguish the two cases:
    \begin{itemize}
        \item If $w_{j+1}=sw_j$, we define \[w_i':=\begin{cases} sw_{i-1},\mbox{ for }2\le i\le j; \\
        w_i, \mbox{ for } j+1\le i\le l. \end{cases}\] We note that if $w_{j-1}\red w_j$ by the pair $\{s_{j-1},t_{j-1}\}$, then by Remark \ref{oss}(6) and (7) $w_j'=sw_{j-1}\red sw_j=w_{j+1}$ by the same pair. Hence, it follows that $w$ is reducible to $v$ with sequence \[w_0'=w\red w_1'\red w_2'\red \cdots \red w_l'=v.\]
        \item If $w_{j+1} =w_jt$, we define \[w_i':=\begin{cases}
            sw_{i-1}, \mbox{ for } 2\le i \le j; \\
            sw_{i}t, \mbox{ for } j+1\le i \le l. 
        \end{cases} \] Note that, for $i\le j-1$, $w_{i}'\red w_{i+1}'$ as above. Recall that $z$ is a prefix of $w_{j}$ but not of $w_{j+1}$. Hence, $w_j=zu_j=u_jz$, thus for all $x\in supp(u_j)$, $x$ commute with both $s$ and $t$. Therefore, $w_{j+1}=u_js=su_j$ and there exists a sequence \[u_j\red u_{j+1}\red \cdots \red u_{l-1}\in \I(W),\] such that $w_i=su_{i-1}$, so $u_{i-1}$ commutes with both $s$ and $t$ for all $j+1\le i\le l$. Moreover, if $w_{j-1}\red w_j=stu_j$ by the pair $\{s_{j-1},t_{j-1}\}$, then $w_{j}'=sw_{j-1}\red w_{j+1}'=sw_{j+1}t$, since in the case $s_{j-1}t_{j-1}$ is a prefix of $w_{j-1}$ then $sw_{j-1}\red sw_{j}=tu_{j}=sw_{j+1}t$ by Remark \ref{oss}(6). Otherwise, if $s_{j-1}t_{j-1}$ is a suffix of $w_{j-1}$, then $sw_{j-1}\red sw_{j}=tu_{j}=sw_{j+1}t$ by Remark \ref{oss}(7). Furthermore, $w_{i}'=sw_it=u_{i-1}t\red u_it=sw_{i+1}t=w_{i+1}'$
        By calling $v':=w'_l=tu_{l-1}\in\I(W)$, we have that $w$ is star reducible to $v'$ with sequence:
       \[w_0'=w\red w_1'\red w_2'\red  \cdots \red w_l'=v'.\]
    \end{itemize}
    Assume now $|S|>2$, if $supp(v)$ is complete, then necessarily we have the first case since in the second one $v=su_{l-1}=u_{l-1}s$ which is impossible, therefore $v=v'$.

    The assertions for weak star reducibility can be proved similarly, considering the prefix $z=\xi_{s,t}$ instead of $z=st$. 
\end{proof}

\begin{Theorem}\label{theorem:starope}\
    \begin{itemize}    
        \item[(a)] If $v\in \fc(W)$ is left (respectively, right) (weak) irreducible and obtained from $w$ by a series of left (respectively, right) (weak) reductions, then $v$ is unique.
        \item[(b)] Let $w,v,v'\in \fc(W)$ with $v,v'$ (weak) irreducible. If $w$ is (weak) reducible to $v$ and $v'$ with respectively sequences \begin{align} w_0&=w\red w_1\red w_2\red \cdots \red w_l=v, \label{eq:seq1}\\
    w_0'&=w\red w_1'\red w_2'\red \cdots \red w_k'=v' \label{eq:seq2};
    \end{align} then $k=l$. 
        \item[(c)] Assume $|S|>2$. If $w$ is (weak) reducible to $v$ (weak) irreducible and $supp(v)$ is complete, then $v$ is unique. 
    \end{itemize}
\end{Theorem}

\begin{proof}
   We first prove the statement for star reducibility.
    \begin{itemize} 
        \item[(a)] %We consider a reduction sequence $w_0=w\mapsto w_1\mapsto w_2\mapsto \cdots \mapsto w_l=v$ with $v=u_0\cdots u_m$ its CFNF.
        %We prove the result by induction on $l$.
        Suppose $w$ reducible to $v$ by a series of left reductions with sequence \[w_0=w\red w_1\red \cdots \red w_l=v.\]
        First we show the following claim by induction on $l$:  \begin{equation}              \label{eq:suff}
                     \mbox{if }x\in \su(w) \mbox{ is left irreducible, then }x\in \su(v).
                    \end{equation} If $l=0$, then $w$ is left irreducible, so the claim trivially holds. So suppose $l\ge 1$ and $w_1=sw$ with $st$ a prefix of $w$ with $m_{s,t}\ge 3$. We have that $x\in \su(w_1)$, otherwise $st$ would be a prefix of $x$ which is not possible since $x$ is left irreducible. Therefore, since $w_1$ is reducible to $v$ by a series of left reductions, by inductive hypothesis $x\in \su(v)$. 
    
                    Now, we observe that if $w$ is reducible to $v'$ left irreducible by a series of left reductions, we have that $v'\in \su(w)$, so by \eqref{eq:suff} $v'\in \su(v)$. So, the same argument, switching the role of $v$ and $v'$, shows that $v\in \su(v')$, thus $v=v'$, hence statement (a) holds. 
       \item[(b)] We proceed by induction on $l$. If $l=0$, then $w$ is irreducible and the thesis trivially holds. So suppose $l\ge 1$. We have that $w$ is left or right reducible to $w_1$ with respect to a pair $\{s_1,t_1\}$, then by applying Lemma \ref{lemma:lemmaseq} to the sequence in \eqref{eq:seq2}, there exists $v''\in \I(W)$ such that $w$ is reducible to $v''$ with a sequence \[w_0''=w\red w_1''=w_1\red \cdots \red w_k''=v''.\] Therefore, $w_1$ is reducible to $v''$ and $v$, so by inductive hypothesis we can conclude that $k=l$. 
       \item[(c)] Assume $|S|>2$, let $w$ be reducible to $v\in \I(W)$, such that $supp(v)$ is complete.
       We proceed by induction on $l$. If $l=0$, then $w=v$ and the claim trivially holds. On the other hand, assume $l>0$ and suppose $w$ be reducible to $v'\in \I(W)$ with sequence \[ w_0'=w\red w_1'\red \cdots \red w_l'=v'.\] Applying Lemma $\ref{lemma:lemmaseq}$, we have that $w$ is reducible to $v$ by a sequence \[w_0''=w\red w_1''=w_1'\red \cdots \red w_l''=v.\] Then, $w_1'$ is reducible to $v'$ and $v$, so by inductive hypothesis, $v=v'$.
    \end{itemize}  
    The statements for weak star reducibility can be proved similarly, in particular in the proof of (a) it is sufficient to work with prefix $\xi_{s,t}$ instead of prefix $st$. On the other hand, for proofs (b) and (c) one has only to apply the weak star version of Lemma $\ref{lemma:lemmaseq}$. 
\end{proof}

\begin{Example}\
\begin{itemize}
    \item[(1)] From Theorem \ref{theorem:starope}, the irreducible element obtained from $w$ may depend on the choice of the reduction sequence when both left and right reductions are used. For instance, let $w_1\in \widetilde{D}_{7}$ be as in Example \ref{exheap}. Then $w_1$ is reducible to both  $v_1=s_0s_3s_6s_{7}$ and $v_2=s_1s_4s_6s_7$ in $\I(\widetilde{D}_{7})$ via the following sequences:
\begin{align*}
    w_1&\red s_4w_1\red s_4w_1s_1 \red s_5s_4w_1s_1 \red s_5s_4w_1s_1s_2 \red v_1=s_5s_4w_1s_1s_2s_4=s_0s_3s_6s_{7};\\
     w_1&\red s_0w_1\red s_4s_0w_1 \red s_3s_4s_0w_1 \red s_5s_3s_4s_0w_1 \red v_2=s_2s_5s_3s_4s_0w_1=s_1s_4s_6s_7.     
\end{align*} 

\item[(2)] Recall that irreducibility and weak irreducibility are different in non-simply laced Coxeter groups. Consider $w_2\in \fc(\widetilde{B}_{6})$ in Example \ref{exheap}, then $w_2$ is weak reducible to $v_2=(s_1s_3s_5)(s_2s_4s_6)(s_0s_3s_5) \in \IRc(\widetilde{B}_{6})$ by a sequence:
\begin{align*}
    w_2\red s_3w_2\red s_4s_3w_2 \red s_2s_4s_3w_2 \red s_2s_4s_3w_2s_6 \red v_2=s_2s_4s_3w_2s_6s_2.
\end{align*}
%with $v_2=(s_1s_3s_5)(s_2s_4s_6)(s_0s_3s_5)$, it is unique by Theorem $\ref{theorem:starope}$ (3). 
Moreover, $v_2$ is reducible to $s_2s_4s_6\in \I(\widetilde{B}_6)$. Thus, $w_2$ is also reducible to $s_2s_4s_6$. 
\end{itemize}
\end{Example}

%%%%%%%%%%%%%%%%%%%%%%%%%%%%%%%%%%%%%
\section{Irreducibility in $\fc(\D)$} \label{se2} 
%%%%%%%%%%%%%%%%%%%%%%%%%%%%%%%%%%%%

In this section, we classify the irreducible elements in a Coxeter system of type $\D$. Recall that, since $\D$ is simply laced, star reducibility and weak star reducibility coincide.

\begin{Definition}
\label{occorrenze}
Let $w \in \fc(\D)$. We denote by
\begin{align*}
%f_\bullet(w)&=\text{ maximum number of factors $s_0s_1$ appearing in $\mathbf{w}\in \mathcal{R}(w)$;}\\
f_\bullet(w)&:=\text{max$\{\#$ of occurrences of the factor $s_0s_1$ in $\mathbf{w} \mid  \mathbf{w} \in \mathcal{R}(w)\};$}\\
f_\circ(w)&:=\text{max$\{\#$ of occurrences of the factor $s_{n+1}s_{n+2}$ in $\mathbf{w} \mid  \mathbf{w} \in \mathcal{R}(w)\}.$}
\end{align*}
\end{Definition}

\begin{Remark}
\label{rem:occirr}
Let $w\in \fc(\D)$, then
\begin{itemize}
    \item[(1)] $f_\bullet(w) = f_\bullet(w^{-1})$ and $f_\circ(w) = f_\circ(w^{-1})$;
    \item[(2)] if $ w \red v  $, then $f_\bullet(v)=f_\bullet(w)$ and $f_\circ(v)=f_\circ(w)$.
\end{itemize}
 
\end{Remark}

\begin{Example}
\label{exocc}
Let $w,v\in \fc(\widetilde{D}_{10})$, $w=s_0s_3s_5s_9s_1s_2s_4s_6s_8s_3s_5s_7s_{10}$ and $v=s_0s_5s_{9}s_{1}s_{10}$. Then we can choose
\begin{align*}   \mathbf{w}=s_3s_5s_9(s_0s_1)s_2s_4s_6s_8s_3s_5s_7s_{10} \quad\mbox{and}\quad
    \mathbf{v}=(s_0s_1)s_5(s_9s_{10}).
\end{align*}
Hence, $f_\bullet(w)=1$, $f_\circ(w)=0$ and $f_\bullet(v)=f_\circ(v)=1$.
\end{Example}

\begin{Notation}
\label{helpestremi}
Since the generators $s_0$ and $s_1$ play equivalent roles, we introduce the notation $s_\bullet$ to denote either $s_0$ or $s_1$, and write $\{\s,\ti\}=\{s_0,s_1\}$. In the same way, we denote by $\sn$ either $s_{n+1}$ or $s_{n+2}$, and $\{\sn,\tn\}=\{s_{n+1},s_{n+2}\}$. 
\end{Notation}

We denote by $\fzp,\fzs$ the elements in $\fc(\D)$
\begin{align}
    \fzp=s_2s_3\cdots s_ns_{n+1}s_{n+2}\quad \mbox{and} \quad   \fzs=s_ns_{n-1}\cdots s_{2}s_1s_0.\label{eq:AB}
\end{align}

\begin{Definition}
\label{def:zigzag}
An element $w\in \fc(\D)$  is called a \textit{complete zigzag}  if it admits a reduced expression of one of the following forms:
\begin{enumerate}
    \item $s_0s_1(AB)^kA^h$,
    \item $s_{n+2}s_{n+1}(BA)^kB^h$, 
\end{enumerate}
where $k\ge 0$, $h\in \{0,1\}$ and $k+h >0$.

\end{Definition}
Denote by $ \zz(\D) \subset \fc(\D) $ the set of complete zigzag elements, and note that they are irreducible (see some examples in Figure \ref{zigzagfig}).

\begin{figure}[h!]
    \centering
    \includegraphics[width=1.0\linewidth]{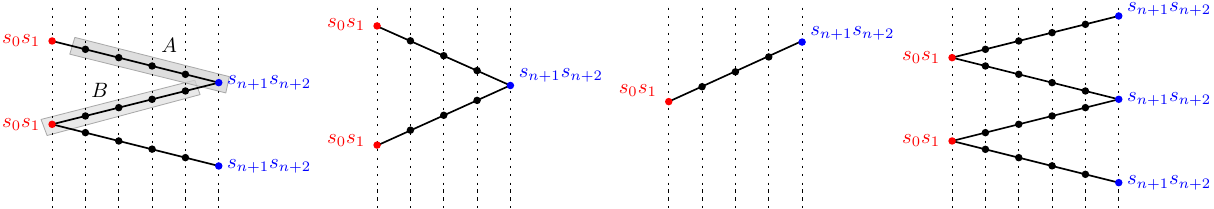}
    \caption{Examples of heaps of complete zigzags.}
    \label{zigzagfig}
\end{figure}

We observe that, if $w$ is a complete zigzag, then $\ld(w)$ and $\rd(w)$ coincide with either $\{s_0,s_1\}$ or $\{s_{n+1},s_{n+2}\}$. Moreover, we have that $f_\bullet(w),f_\circ(w) \ge 1$. In particular, $\ell(w)=(2k+h)(n+1)+2$, and $w$ is uniquely determined by $\ell(w)$ and $\ld(w)$. %in particular, we denote $w$ by $ \xi^{h}_{s_0s_1} $ when $\ld(w)=\{s_0,s_1\}$, and by $\xi^{h}_{s_{n+1}s_{n+2}} $ when $\ld(w)=\{s_{n+1},s_{n+2}\}$. 
Note that $\ld(w)=\rd(w)$ if and only if $h=0$. Moreover, if $w\in \zz(\D)$, then \[ f_\bullet(w)+f_\circ(w)=2k+h+1\,\, \mbox{ and }\,\, \left|f_\bullet(w)-f_\circ(w)\right|\le 1.\] Observe also that $w\in \zz(\D)$ if and only if $w^{-1}\in \zz(\D)$.

\begin{Definition}
\label{weakzigzag}
An element $w \in \fc(\D)$ is called a \textit{weak zigzag} if it is a factor of a complete zigzag and $w\notin \TTC(\D)$. 
\end{Definition}

The family of weak zigzags given in Definition \ref{weakzigzag} does not coincide neither with the zigzags introduced in \cite[Definition 3.1]{BJNFC}, nor with the pseudo zigzags defined in \cite[Definition 1.8]{BFS}.
Note that $w$ is a weak zigzag if and only if $w^{-1}$ is a weak zigzag. Moreover, a weak zigzag $w$ is irreducible if and only if $w$ is a complete zigzag.

\begin{Example}
Some complete zigzags in $\fc(\widetilde{D}_6)$ are:
\begin{enumerate}
    \item $w_1=(s_0s_1)s_2s_3s_4(s_5s_6)$;
    \item $w_{2}=(s_5s_6)s_4s_3s_2(s_0s_1)s_2s_3s_4(s_5s_6)$;
    \item $w_3=(s_0s_1)s_2s_3s_4(s_5s_6)s_4s_3s_2(s_0s_1)s_2s_3s_4(s_5s_6)$.
\end{enumerate}
While some weak zigzags in $\fc(\widetilde{D}_6)$ are:
\begin{enumerate}
    \item $v_1=(s_0s_1)s_2s_3$, which is a prefix of $w_1$;
    \item $v_2=s_2(s_0s_1)s_2s_3s_4(s_5s_6)$, which is a suffix of $w_{2}$;
    \item $v_3=s_3s_4(s_5s_6)s_4s_3s_2(s_0s_1)s_2s_3s_4s_6$, which is a factor of $w_{3}$.
\end{enumerate}
\end{Example}

\begin{Definition}
\label{caramella}
Let $ w \in \fc(\D) $, $ n $ even, and let $ w = u_0 \cdots u_m $ be its CFNF with $ m \ge 2 $. We say that $ w $ is a \textit{candy} if $ m $ is even and the following conditions hold:
\begin{itemize}
    \item $ supp(u_{2i}) = \{s_3, s_5, \ldots, s_{n-1}\} \cup \left\{ x_i, y_i \right\} $, $ x_i \in \left\{s_0, s_1 \right\} $ and $ y_i \in \left\{ s_{n+2}, s_{n+1} \right\} $ such that $ x_i \neq x_{i+1} $ and $ y_i \neq y_{i+1} $, for every $ 0 \le i \le \frac{m}{2} $;
    \item $ supp(u_{2i+1}) = \{s_2, s_4, \ldots, s_n\} $, for every $ 0 \le i < \frac{m}{2} $.
\end{itemize}

\end{Definition}
Denote by $ \car(\D) \subset \fc(\D) $ the set of candy elements. We set $\car (\D)=\emptyset$ when $n$ is odd. All candy elements are irreducible. 
 
\begin{figure}[h!]
    \centering    
    \includegraphics[width=0.45\linewidth]{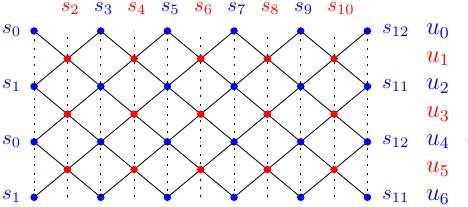}
    \caption{Example of a heap of a candy element in $\fc(\widetilde{D}_{12})$.}
    \label{candyfig}
\end{figure}

Any $w\in \car(\D)$ is uniquely determined by $ x_0, y_0 $ and $ m $; furthermore we have that $\ld(w)=\rd(w)$ if and only if $m/2$ is even. % We denote an element of this type by $ \kappa^m_{x_0, y_0} $. 
Moreover, $f_\bullet(w)=f_\circ(w)=0$. Observe that $w\in \car (\D)$ if and only if $w^{-1}\in \car (\D)$.

\begin{Lemma}
\label{lemma:condzigzag}
Let $w\in \fc(\D)$ be such that:
\begin{align*}
\ld(w) &= \left\{ s_0, s_1 \right\} \text{ or } \ld(w) = \left\{ s_{n+2}, s_{n+1} \right\}\\
(\mbox{respectively, } \rd(w) &= \{ s_0, s_1 \} \text{ or } \rd(w) = \{ s_{n+2}, s_{n+1}\}) .
\end{align*}
Then either $w\in \TTC(\D)$ or $w$ is a weak zigzag. Moreover, $w$ is left (respectively, right) irreducible. 

%If $\rd(w)=\{s_0,s_1\}$ or $\rd(w)=\{s_{n+1},s_{n+2}\}$, then the same result holds with $w$ right irreducible.
\end{Lemma}

\begin{proof}
Without loss of generality suppose that $\ld(w) = \left\{ s_0, s_1 \right\}$. Suppose $w\notin \TTC(\D)$ and let $w=u_0\cdots u_m$ be its CFNF, so $m\ge 1$. If $supp(u_i)\ne \emptyset$ then $u_i=s_{i+1}$ for $1\le i\le n-1$ or $supp(u_i)\subseteq \{s_{n+1}, s_{n+2}\}$ for $i=n$. Therefore by CFNF and fully commutativity we have: \[w=\begin{cases}
    s_0s_1s_2\cdots s_{m}s_{m+1}, & m<n, \\ s_0s_1s_2\cdots s_n u_n, & m=n \mbox{ and } u_n\in \{s_{n+1},s_{n+2}\}, \\ s_0s_1s_2\cdots s_ns_{n+1}s_{n+2}P, & m\ge  n \mbox{ and } P\in \pr((\fzs \fzp)^k)\mbox{, } k\ge 1
    \end{cases}\] 
%Then if $m>n$, then $u_{n+1}=s_{n+1-i}$ and $u_{2n}\subseteq \{ s_0,s_1 \}$, so $w$ will be of the form: 
%\[w=\begin{cases}
 %   s_0s_1s_2\cdots s_ns_{n+1}s_{n+2} s_n \cdots s_{m-i} s_{m+1-i}, & n<m<2n \\ s_0s_1s_2\cdots s_ns_{n+1}s_{n+2} s_n \cdots s_{2} \s, & m=2n \mbox{ and } u_{2n}=\s \\ s_0s_1s_2\cdots s_ns_{n+1}s_{n+2} s_n \cdots s_{2}s_0s_1\cdots, & m\ge 2n \mbox{ and } u_{2n}=s_0s_1
 %   \end{cases}\] 
 where $A$ and $B$ are defined in \eqref{eq:AB}. In conclusion, either $w\in \TTC(\D)$ or $w$ is a weak zigzag, and in both cases $w$ is left irreducible.
\end{proof}

\begin{Corollary}
\label{corollary:precompzigzag}
Let $w\in \fc(\D)$. 
\begin{itemize}
    \item[(a)] If $x\in \pr(w)\cup \su(w)$ such that $x\in \zz(\D)$, then $w$ is a weak zigzag. 
    \item[(b)] If $w$ is reducible to $v\in \zz(\D)$, then $w$ is a weak zigzag. 
\end{itemize}

\end{Corollary}

\begin{proof}
Without loss of generality, we assume $w=xy$ reduced product and $\rd(x)=\{s_0,s_1\}$. Since $x\in \zz(\D)$, $\ld(y)\subseteq \{s_2\}$. Hence, $\ld(x)=\ld(w)$ and $w$ is a weak zigzag since $w\notin \TTC(\D)$, and by Lemma \ref{lemma:condzigzag}, so (a) is proved.

Now assume that $w$ is reducible to $v\in\zz(\D)$, then we can factorize $w$ as $w=v'vv''$ with $\ell(w)=\ell(v')+\ell(v)+\ell(v'')$. 
By part (a), both $v'v$ and $vv''$ are weak zigzags. Since $w\in \fc(\D)$  and the factorization is length-additive, it follows that 
$w$ itself is a weak zigzag.
\end{proof}

\begin{Lemma}
\label{lemma:occorrenzeuno}
Let $w\in \fc(\D)$ left irreducible and not a weak zigzag. 
\begin{enumerate}
    \item[(a)] If $f_\bullet(w)\ne 0$, then $s_0,s_1\in \ld(w)$ and $f_\bullet(w)=1$;
    \item[(b)] If $f_\circ(w)\ne 0$, then $s_{n+1},s_{n+2}\in \ld(w)$ and $f_\circ(w)=1$.
\end{enumerate}
\end{Lemma}

\begin{proof}
We prove (a), since (b) is similar. 
We have that $w$ can be factorized as a reduced product $w=vv'$ with $\rd(v)=\{s_0,s_1\}$ and $f_\bullet(v')=0$. Assume $v\notin \TTC(\D)$, then by Lemma \ref{lemma:condzigzag} we have that $v$ is a weak zigzag and it is right irreducible. Moreover, since $w$ is left irreducible and $v\in \pr(w)$, it follows that $v$ is a complete zigzag. But this is a contradiction by Corollary \ref{corollary:precompzigzag}, because $w$ is not a weak zigzag. Therefore, $v\in \TTC(\D)$, so $v=s_0s_1$ and $s_0,s_1\in \ld(w)$. Moreover, since $f_\bullet(v')=0$, it follows that $f_\bullet(w)=1$. 
\end{proof}

\begin{Lemma}
\label{secondlemmanormalform}
Let $w\in \fc(\D)$, $w=u_0\cdots u_m$ be its CFNF, and $2\le i\le n$. Let $k\ge1$ be the minimum index such that $s_i\in supp(u_k)$.\\
If $f_\bullet(w)=0$, then
\begin{enumerate}
    \item[(a)] if $i=2$ and $\s\in supp(u_{k-1})$, then $k=1$;
    \item[(b)] if $i\ge 3$ and $s_{i-1}\in supp(u_{k-1})$, then there exists $z\in \pr(w)$ with $z\in \su(\s s_2\cdots s_{i-1})$ and $\ell(z)=k$.
\end{enumerate}
Similarly, if $f_\circ(w)=0$, then
\begin{enumerate}
    \item[(c)] if $i\le n-1$ and $s_{i+1}\in supp(u_{k-1})$, then there exists $z\in \pr(w)$ with $z\in \su(\sn s_n\cdots s_{i+1})$ and $\ell(z)=k$;
    \item[(d)] if $i=n$ and $\sn \in supp(u_{k-1})$, then $k=1$.
\end{enumerate}
\end{Lemma}

\begin{proof}
We prove (a) and (b), since (c) and (d) are analogous. 

Set $i=2$ and suppose $\s \in supp(u_{k-1})$. If $k\ge 2$, then $s_2\in supp(u_{k-2})$ by Theorem \ref{normalform}(d), but this is a contradiction. Thus, $k=1$.
 
Set $i\ge 3$, assume $s_{i-1}\in supp(u_{k-1})$, and let $w_{k-1}:=u_0\cdots u_{k-1}$. Since $s_{i-1}\in \rd(w_{k-1})$, it is possible to factorize $w_{k-1}$ as a reduced product $w_{k-1}=zv$, with $\rd(z)=\{s_{i-1}\}$, $s_{i}\notin supp(z)$, and $s_{i-1}\notin supp(v)$. Let $z^{-1}=u'_0\cdots u'_l$ be its CFNF, then $u'_0=s_{i-1}$. If $supp(u_1')\ne \emptyset$, then $u_1'=s_{i-2}$ by $s_i\notin supp(z)$ and Theorem \ref{normalform}(d). In general, if $supp(u_{j}')\ne \emptyset$ with $i-1-j\ge 1$, we have that: 
    \[u'_j=\begin{cases}
    s_{i-1-j}, & i-1-j\ge 2; \\ \s, & i-1-j=1.
    \end{cases}\] 
    If $l\ge i-1$, then $u_{i-2}'=\s$ and $u'_{i-3}=u_{i-1}'=s_2$, that is a contradiction. Thus, $l\le i-2$. Therefore, $z^{-1}$ is a prefix of $s_{i-1}\cdots s_2\s$, thus $z$ is a suffix of $\s s_2 \cdots s_{i-1}$. Furthermore, since $z$ is prefix of $w_{k-1}$, $s_{i-1}\in supp(u_{k-1})$ and $s_{i-1}\notin supp(v)$, by Lemma $\ref{lemmaunitarissimo}$, $\ell(z)=k$.
 
\end{proof}

\begin{Example}
\label{exseclemmanf}
Let $w=u_0u_1u_2u_3\in \fc(\widetilde{D}_{11})$ with $u_0=s_0s_6s_8s_{10}s_{11}$, $u_{1}=s_2s_5s_7s_9$, $u_2=s_3s_6s_8$, and $u_3=s_4s_7$ its CFNF. Then $f_\bullet(w)=0$, $s_4\in supp(u_3)$, $s_4\notin supp(u_{0}u_1u_2)$, and $s_3\in supp(u_2)$. Note that there exists a prefix $z$ of $w$ such that $\rd(z)=\{s_3\}$: \[w=\underbrace{s_0s_2s_3}_{z}s_6s_8s_{10}s_{11}s_5s_7s_9s_6s_8s_4s_7.\] Observe that $z$ is of the form described in Lemma $\ref{secondlemmanormalform}$ and $\ell(z)=3$. 
\end{Example}

\begin{Lemma}
\label{cornormalform}
Let $w\in \fc(\D)$ and $w=u_0\cdots u_m$ be its CFNF, the following hold.
\begin{enumerate}
    \item[(a)] Let $k$ be the minimum index such that $supp(u_k)\cap \{s_0,s_1\}\ne \emptyset$. If $f_\circ(w)=0$, then there exists $z\in \pr(w)$ with $z\in \su (\sn s_n\cdots s_2\s)$ and $\ell(z)=k+1$.
    \item[(b)] Let $k$ be the minimum index with $supp(u_k)\cap \{s_{n+1},s_{n+2}\}\ne \emptyset$. If $f_\bullet(w)=0$, then there exists $z\in \pr(w)$ with $z\in \su(\s s_2\cdots s_n\sn)$ and $\ell(z)=k+1$.
\end{enumerate}
\end{Lemma}

\begin{proof}
We prove only (a) since (b) is analogous. If $k=0$, then  $\s$ is a prefix of $w$. Otherwise, assume $k\ge 1$ and let $w_k:=u_0\cdots u_k$. Since since $ \s\in supp(u_k)$ and $supp(u_j)\cap \{s_0,s_1\}=\emptyset$ for all $0\le j\le k-1$, it is possible to factorize $w_k$ as a reduced product $w_k=zv$, with $\rd(z)=\{\s\}$ and $\s \notin supp(v)$. Let $z^{-1}=u'_0\cdots u'_l$ be its CFNF, then necessarily $u'_0=\s$ and if $supp(u'_i)\ne \emptyset$ we have that
\[u'_i=\begin{cases}
    s_{i+1}, & 1\le i\le n-1;\\ \sn, & i=n.
    \end{cases}\] 
Since $f_\circ(w)=0$, it follows that $l\le n$. Otherwise if $l>n$, then $u'_{n+1}=s_n$ by Theorem \ref{normalform}(d), but this cannot be because $u_n=\sn$ and $u_{n-1}=s_n$, so $s_n\sn s_n$ would be a factor of $w$. In conclusion, $z^{-1}$ is a prefix of $\s s_2\cdots s_n \sn$, thus $z\in \pr(w_k)\subseteq \pr(w)$ is a suffix of $\sn s_n\cdots s_2 \s$. Furthermore, by Lemma $\ref{lemmaunitarissimo}$, $\ell(z)=k+1$.
\end{proof}

\begin{Example}
\label{excorlemmanf}
Let $w=u_0u_1u_2u_3\in \fc(\widetilde{D}_{10})$ such that $u_0=s_4s_6s_9$, $u_{1}=s_3s_5s_8$, $u_2=s_2s_4s_7s_{10}$, $u_3=s_0s_1s_3s_8$. Then $f_\circ(w)=0$ and $supp(w)\cap \{s_0,s_1\}\ne \emptyset$. Hence there exists a prefix $z$ of $w$ such that $\rd(z)=\{s_0\}$: \[w=\underbrace{s_4s_3s_2s_0}_{z}s_6s_9s_5s_8s_4s_7s_{10}s_1s_3s_8.\] Observe that $z$ is a prefix of $w$ of the form described in Lemma $\ref{cornormalform}$ and $\ell(z)=4$.
\end{Example}

\begin{Lemma}
\label{lemma:caruno}
Let $w\in \I(\D)\setminus \zz(\D)$ with complete support and $w=u_0\cdots u_m$ be its CFNF. The following hold:
\begin{enumerate}
    \item[(a)]  $m\ge 1$;
    \item[(b)]  $f_\bullet(w)=f_\circ(w)=0$;
    \item[(c)]  $\left |supp(u_0)\cap \{s_0,s_1\}\right |=1$;
    \item[(d)]  $\left |supp(u_0)\cap \{s_{n+1},s_{n+2}\}\right |=1$.
\end{enumerate}

\end{Lemma}

\begin{proof}

If $m=0$, then $w=u_0$ would be a product of commuting generators and it could not have complete support. Hence, $m\ge 1$ and (a) is proved. 

We prove by contradiction that $f_\bullet(w)=0$; the case with $f_\circ(w)$ is analogous. Assume $f_\bullet(w)\ne 0$, then by applying Lemma $\ref{lemma:occorrenzeuno}$ to $w$ and $w^{-1}$, we have that $s_0,s_1\in \ld(w)\cap \rd(w)$ and $f_\bullet(w)=1$.  This implies that $s_2\notin supp(w)$. But this a contradiction since $supp(w)$ is complete. 

To prove (c), by point (b) it is sufficient to show the following claim: 
$$ supp(u_0)\cap \{s_0,s_1\}\ne \emptyset.$$ 
We have that $f_\circ(w)=0$, and since $s_0,s_1\in supp(w)$, by Lemma $\ref{cornormalform}$, there exists $z\in \pr(w)$ with \[z\in \su(\sn s_n \cdots s_2 \s)\] and $\ell(z)\ge 1$. If $\ell(z)\ge 2$, then $z$ is a left reducible prefix of $w$, but this is a contradiction, since $w\in \I(\D)$. Therefore, $z=\s$, so $supp(u_0)\cap \{s_0,s_1\}\ne \emptyset$. In conclusion, by $f_\bullet(w)=0$, $\left |supp(u_0)\cap \{s_0,s_1\}\right |=1$. The proof of (d) is analogous of (c). 

\end{proof}

\begin{Lemma}
\label{cardue}
Let $w\in \I(\D)\setminus \zz(\D)$ with complete support, and let $w=u_0\cdots u_m$ be its CFNF. Then, the following hold:
\begin{enumerate}
    \item[(a)]  $supp(u_0)=\{\s\} \cup \{s_3, s_5,\ldots, s_{n-1}\} \cup \{\sn\}$;
    \item[(b)] $supp(u_1)= \{s_2, s_4, \ldots, s_n\}$.
\end{enumerate}
In particular, $n$ is even.
\end{Lemma}
%bisogna fare un lemma
\begin{proof}
By Lemma $\ref{lemma:caruno}$, we have $m\ge 1$, $\s,\sn\in supp(u_0)$, and $f_\bullet(w)=f_\circ(w)=0$. Consider the sets 
\begin{align*}
    U_0&:=\{2j+1 \mid s_{2j+1}\notin supp(u_0), \, 1 \leq 2j+1 \leq n\} \mbox{ and }\\
    U_1&:=\{ 2j \mid s_{2j}\notin supp(u_1), \, 1\le 2j\le n\}.
\end{align*}
Set $h:=\min( U_0\cup U_1)$.
First, observe that $h\in U_1$, otherwise if $h\in U_0$, then by minimality, $s_{h-2}s_{h-1}$ or $\s s_2$, for $h=3$, would be a prefix of $w$, which is not possible. Also, note that $\s \in supp(u_0)$ and if $h> 3$,  $s_{h-1}\in supp(u_0)$ by minimality, hence $s_h\notin supp(u_0)$ since $u_0$ is a product of commuting generators. Thus let $k\geq 2$ be the minimum index such that $s_{h}\in supp(u_k)$. By Lemma $\ref{secondlemmanormalform}$, there exists $z\in \pr(w)$ with
\[z\in \su(\s s_2\cdots s_{h-1}) \quad \mbox{or} \quad z\in \su(\sn s_n\cdots s_{h+1})\] and $\ell(z)=k\ge 2$. In both cases we have a contradiction. In the first one, $z$ cannot be a prefix since $s_{h-1}\in \ld(w)$ by minimality of $h$; in the second, $z$ cannot be a prefix since $w\in \I(\D)$. 
Hence $U_0\cup U_1=\emptyset$ from which the result follows.
\end{proof}

\begin{Remark}
\label{rem:stared}
    In \cite[Theorem 6.3]{GreenStar}, Green shows that Coxeter groups of finite type $D$ are star reducible. This means that $w\in \fc(D_n)$ is irreducible if and only if $w$ is a product of commuting generators. 
    In our setting, if $w\in \fc(\D)$ and $supp(w)$ is not complete, then $w$ can be considered as a FC element or a product of FC elements in a Coxeter group of finite type $D$. Hence, if $w\in \I(\D)$ with not complete support, then $w\in \TTC(\D)$.
\end{Remark}

\begin{Proposition}
\label{cartre}
Let $w\in \I(\D)\setminus \zz(\D)$ with complete support. Then $w\in \car(\D)$. 
\end{Proposition}

\begin{proof}
Let $w=u_0\cdots u_m$ be its CFNF. By Lemma $\ref{cardue}$ we know that $m\ge 1$, $n$ is even and $supp(u_0)$, $supp(u_1)$ are described in points (a) and (b). This implies $m\ge 2$, otherwise $w$ would admit $\s s_2$ as a suffix, which is not possible since $w\in \I(\D)$. Since $w\in \fc(\D)$, by definition of CFNF, we have that
\begin{align}
    supp(u_2)\subseteq \{\ti\} \cup \{s_3, s_5, \ldots, s_{n-1}\} \cup \{\tn\}:=\mathcal{S}. \label{eq:inclD}
\end{align}
Define $w_2:=u_2\cdots u_m$ and note that is in $\I(\D)$. In fact, it is right irreducible since it is a suffix of $w\in \I(\D)$. Moreover, it is also left irreducible because if $st$ with $m_{s,t}=3$ is a prefix of $w_2$, then $s\in supp(u_2)$ and $t\in supp(u_3)$, but by CFNF we have that $supp(u_3)\subseteq supp(u_1)$, so $tst$ would be a factor of $w$, that contradicts the hypothesis of fully commutativity. 

To summarize, we have $w=u_0u_1w_2$, with $w_2=u_2\cdots u_m\in \I(\D)$ and $m\ge 2$. Now we show by induction on $m\ge 2$ that $w\in \car(\D)$.

Suppose $m=2$, then the inclusion in \eqref{eq:inclD} is an equality. Otherwise there would exist $s'\in \mathcal{S}\setminus supp(u_2)$, $s\in supp(u_1)$, $t \in supp(u_2)$ with $m_{s,t}=m_{s',s}=3$, and the suffix $st$ would appear in $w$, that is not possible. Hence $w=u_0u_1u_2\in \car(\D)$; its heap is depicted in Figure \ref{fig:small_candy}.

\begin{figure}[ht]
    \centering
    \includegraphics[width=0.5\linewidth]{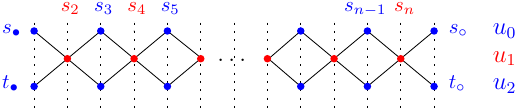}
    \caption{Heap of $u_0u_1u_2$ in the proof of Proposition \ref{cartre}.}
    \label{fig:small_candy}
\end{figure}

If $m>2$, since $w_2$ is irreducible, we know that $supp(w_2)$ is complete, otherwise by Remark \ref{rem:stared}, $w_2\in \TTC(\D)$, but this cannot be since $m>2$. Moreover, since $w_2\in \su(w)$ and by Lemma \ref{lemma:caruno}, we have $f_\bullet(w_2)=f_\circ(w_2)=0$, so $w_2$ is not a complete zigzag. Therefore, by inductive hypothesis, $w_2$ is a candy element, and we can conclude that $w$ is also a candy element, since $w=u_0u_1w_2$.
\end{proof}

\begin{Theorem}[Classification of irreducible elements of type $\D$]\ 
\label{decimo}

\noindent
The set of irreducible elements in $\fc(\D)$ is partitioned in three families as follows:
\begin{align*}
    \I(\D) = \car (\D) \sqcup \zz(\D) \sqcup \TTC(\D).
\end{align*}
If $n$ is odd, $\car(\D)=\emptyset$. 
\end{Theorem}

\begin{proof}
Let $w\in \I(\D)$. If $supp(w)$ is not complete then $w\in \TTC(\D)$, by Remark \ref{rem:stared}. If $supp(w)$ is complete and $w$ is not a complete zigzag, then $n$ is even and $w\in \car(\D)$, by Proposition $\ref{cartre}$. 
\end{proof}

\begin{Corollary}
\label{corollary:corclassification}
    If $w\in \fc(\D)$ is reducible to $v,v'\in \I(\D)$, then $v,v'$ belong to the same family of irreducible elements. In particular, if $v,v'\notin \TTC(\D)$, then $v=v'$. 
\end{Corollary}

\begin{proof}
    If $v\notin \TTC(\D)$, then $supp(v)$ is complete by Theorem \ref{decimo}, so $v=v'$ by Theorem \ref{theorem:starope}(c).
\end{proof}

%%%%%%%%%%%%%%%%%%%%%%%%%%
\section{Irreducibility and weak irreducibility in $\fc(\B)$}\label{se3}
%%%%%%%%%%%%%%%%%%%%%%%%%

As we mentioned above, in the non-simply laced cases the notions of reducibility and weak reducibility do not coincide, so we have that $\IRc(\B)=\IR(\B)\setminus\I(\B)$ is not empty.
In this section we classify both the irreducible and weak irreducible elements of a Coxeter group of type $\B$. The families of irreducible elements that we now list are analogous to those arising in type $\D$, although certain differences occur in the present setting.

\begin{Definition}
\label{pseudocomm}
We define the set of \textit{weak completely commutative} elements $\TTCw(\B)\subseteq \fc(\B)$ by $$\TTCw(\B):=\TTC(\B)\cup\{s_ns_{n+1}v, s_{n+1}s_nv\mid v\in \TTC(\B), s_{n-1},s_n,s_{n+1}\notin supp(v)\}.$$ 

\end{Definition}

Note that $w\in \TTCw(\B)$ is weak irreducible. In particular, $w\in \I(\B)$ if and only if $w\in \TTC(\B)$. 
\smallskip

We denote by $\mathcal{A},\mathcal{B}$ the elements in $\fc(\B)$
\begin{align*}
    \mathcal{A}=s_2s_3\cdots s_ns_{n+1}\quad \mbox{and} \quad   \mathcal{B}=s_ns_{n-1}\cdots s_{2}s_1s_0.\label{eq:AB}
\end{align*}

\begin{Definition}
\label{bcompletezz}
An element $w\in \fc(\B)$  is called a \textit{complete zigzag}  if it admits a reduced expression of one of the following forms:
\begin{enumerate}
    \item $s_0s_1(\mathcal{AB})^k\mathcal{A}^h$,
    \item $s_{n+1}(\mathcal{BA})^k\mathcal{B}^h$, 
\end{enumerate}
where $k\ge 0$, $h\in \{0,1\}$ and $k+h >0$.
\end{Definition}
\begin{figure}[ht]
     \centering
     \includegraphics[width=1\linewidth]{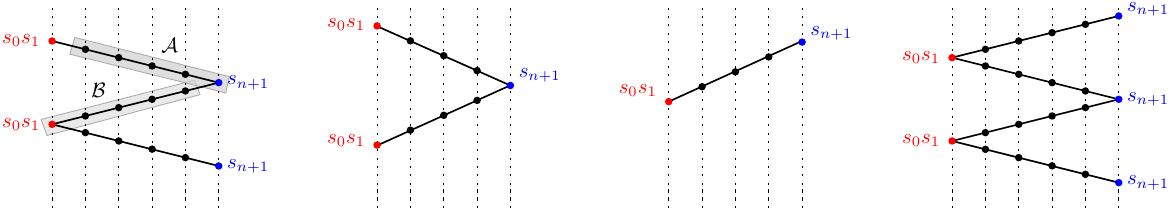}
     \caption{Examples of heaps of complete zigzag elements in type $\B$.}
     \label{fig:CZZB}
 \end{figure}

 Denote by $\zz(\B) \subset \fc(\B)$ the set of complete zigzag elements and note that they are weak irreducible. In particular, $w\in \zz(\B)\cap \I(\B)$ if and only if $w$ is of the form $s_0s_1(\mathcal{AB})^k\mathcal{A}^h$ with $h>0$.

\begin{Definition}
\label{bweakzz}
An element $w \in \fc(\B)$ is called a \textit{weak zigzag} if it is a factor of a complete zigzag and $w\notin \TTC(\B)$. 
\end{Definition}

Note that a weak zigzag $w$ is weak irreducible if and only if $w$ is a complete zigzag.

\begin{Definition}
\label{bcandy}
Let $ w \in \fc(\B) $, $ n $ even, and let $ w = u_0 \cdots u_m $ be its normal form with $ m \ge 2 $. We say that $ w $ is a \textit{candy} if $ m $ is even and:
\begin{itemize}
    \item $ supp(u_{2i}) = \{s_3, s_5, \ldots, s_{n+1} \} \cup \left\{ x_i\right\} $, $ x_i \in \left\{s_0, s_1 \right\} $ such that $ x_i \neq x_{i+1} $, for every $ 0 \le i \le \frac{m}{2} $;
    \item $ supp(u_{2i+1}) = \{s_2, s_4, \ldots, s_n\} $, for every $ 0 \le i < \frac{m}{2} $.
\end{itemize}
Denote by $ \car(\B) \subset \fc(\B) $ the set of candy elements and note that they are irreducible. If $n$ is odd, we set $\car(\B)=\emptyset$.
\end{Definition}

\begin{Definition}
\label{weakcandy}
Let $ w \in \fc(\B) $, $ n $ odd, and let $ w = u_0 \cdots u_m $ be its normal form with $ m \ge 2 $. We say that $ w $ is a \textit{left candy} if $ m $ is even and:
\begin{itemize}
    \item $ supp(u_{2i}) = \{s_3, s_5, \ldots, s_n\} \cup \left\{ x_i\right\} $, $ x_i \in \left\{s_0, s_1 \right\} $ such that $ x_i \neq x_{i+1} $, for every $ 0 \le i \le \frac{m}{2} $;
    \item $ supp(u_{2i+1}) = \{s_2, s_4, \ldots, s_{n+1}\} $, for every $ 0 \le i < \frac{m}{2} $.
\end{itemize}
Denote by $ \carw (\B) \subset \fc(\B) $ denotes the set of left candy elements, in particular $\carw(\B) \subset \IRc(\B) $. If $n$ is even, we set $\carw(\B)=\emptyset$.
\end{Definition}

In the next definition, we introduce an injective map
$$\nmap:\fc(\B) \rightarrow \fc(\D).$$
This map will serve as a tool for classifying the irreducible and weakly irreducible elements of type $\B$, based on the corresponding classification for type $\D$ developed in Section~\ref{se2}.

\begin{Definition}\label{def:mu}   
Let $w \in \fc(\B)$. We define the map $\nmap:\fc(\B) \rightarrow \fc(\D)$ by the following rule:

\begin{itemize}

\item[(1)] If no reduced expression of $w$ contains the factor $s_{n+1}s_{n}s_{n+1}$ or $s_{n}s_{n+1}s_{n}$, consider any $\mathbf{w} \in \mathcal{R}(w)$. Then $\nmap(w)$ is the element in $\fc(\D)$ with reduced expression $\mathbf{w}$.

\item[(2)] Otherwise, choose a reduced expression
$w \in \mathcal{R}(w)$ that contains all possible occurrences of the factors $s_{n+1}s_ns_{n+1}$ and $s_ns_{n+1}s_n$. Then $\nmap(w)$ is the element whose reduced expression is obtained from $\mathbf{w}$ by replacing each occurrence of $s_{n+1}s_ns_{n+1}$ with $s_{n+1}s_ns_{n+2}$, and each occurrence of $s_ns_{n+1}s_n$ with $s_ns_{n+1}s_{n+2}s_n$.

\end{itemize}

\end{Definition} 

The map $\nmap$ is well-defined. Point (1) is immediate, and since in $\fc(\B)$ there are no elements containing both factors $s_{n+1}s_{n}s_{n+1}$ and $s_ns_{n+1}s_n$ (see e.g. \cite[\S 3.2]{BJNFC}), point (2) is also well-posed.

\begin{Lemma}\label{lemma:mubij}
    The map $\nmap$ is injective, and $w\in \I(\B)$ if and only if $\nmap(w)\in \I(\D)$.
\end{Lemma}

\begin{proof}
   Note that if there exists a reducible prefix or suffix in $w$, then it will be transformed by $\nmap$ into a reducible prefix or suffix in $\nmap(w)$. 
   On the other hand, the only reducible prefixes and suffixes that need to be analyzed are $\sn s_n$ and $s_n \sn$. If $\sn s_n$ or $s_n\sn$ is a prefix of $\nmap(w)$, then $\sn=s_{n+1}$ by construction of $\nmap$, so $s_{n+1}s_n$ or $s_ns_{n+1}$ is a prefix of $w$. The analogous holds for the suffix, with the only difference that if $s_n\sn$ is a suffix of $\nmap(w)$, $\sn$ could be equal to $s_{n+1}$ or $s_{n+2}$.  
\end{proof}

\begin{Remark}
\label{remark:weakzigBeD}
    Let $w\in \fc(\B)$, we have that $w$ is a weak zigzag if and only if $\nmap(w)$ is a weak zigzag. Moreover, if $w\in \IR(\B)$ and $\nmap(w)$ is a weak zigzag, then $w\in \zz(\B)$. Finally, $w$ is a candy in $\fc(\B)$ if and only if $\nmap(w)$ is a candy in $\fc(\D)$. 
\end{Remark}

The next theorem follows directly by Theorem \ref{decimo}, Lemma \ref{lemma:mubij} and Remark \ref{remark:weakzigBeD}.

\begin{Theorem}\label{theorem:classredB}
The set of irreducible elements in $\fc(\B)$ is partitioned as follows
    $$\I(\B)=\TTC(\B)\sqcup \car(\B)\sqcup \zz_\bullet(\B),$$
where $\zz_\bullet(\B):=\{w\in \zz(\B)\mid \ld(w)=\rd(w)=\{s_0,s_1\}\}$ and if $n$ is odd, $\car(\B)=\emptyset$.
\end{Theorem}

\smallskip

\begin{Remark}\label{rem:prefirrB}
    By Definitions \ref{def:star-red} and \ref{irr}, if $w\in \IRc(\B)$, then $w$ has a prefix or suffix $st$, with $m_{s,t}\ge3$, such that $tst\in \fc(\B)$. This necessarily implies that $\{s,t\}=\{s_n, s_{n+1}\}$. Indeed, if $w$ has a different prefix or suffix of the form $st$, with $m_{s,t}\geq 3$, then $m_{s,t}=3$ and $sts\notin \fc(\B)$.
\end{Remark}

%In the next lemma, we show that if $w\in \IRc(\B)$ has complete support, then the only prefix (respectively, suffix) that could occur in $w$ is of the form $s_ns_{n+1}$ (respectively, $s_{n+1}s_n$).

\begin{Lemma}
\label{lemma:lemmaBuno}
    Let $w\in \IR(\B) \setminus \zz(\B)$ with complete support. Then we have:
    \begin{itemize}
        %\item[(a)] $s_{n+1}s_n\notin \pr(w) \cup \pr(w^{-1})$.
        \item[(a)] if $w\in \IRc(\B)$, then $s_{n}s_{n+1}\in \pr(w) \cup \pr(w^{-1})$;
        \item[(b)] $f_\bullet(\nmap(w))=f_\circ(\nmap(w))=0$;
        \item[(c)] $\left | \ld(w)\cap \{s_0,s_1\}\right |=1$. 
    \end{itemize}
\end{Lemma}

\begin{proof}    
    %, then it follows that $s_{n+1}s_n\notin \pr(w)\cup \pr(w^{-1})$. 
    %By Remark \ref{rem:prefirrB}, $w$ has a prefix or suffix of the form $s_ns_{n+1}$ or $s_{n+1}s_n$. 
    For proving (a) we proceed by contradiction, so assume that $s_{n}s_{n+1}\notin \pr(w)\cup \pr(w^{-1})$. Therefore, by Remark \ref{rem:prefirrB}, $w$ has a prefix or suffix of the form $s_{n+1}s_{n}$ or $s_{n}s_{n+1}$. Without loss of generality, we assume $s_{n+1}s_n \in \pr(w)$. We define $v:=s_{n+2}\nmap(w)$ if $s_ns_{n+1}\notin \su(w)$ or $v:=s_{n+2}\nmap(w)s_{n+2}$ if $s_ns_{n+1}\in \su(w)$. In both cases, $v\in \fc(\D)$ and note that $v\in \I(\D)$ and $s_{n+1},s_{n+2}\in \ld(v)$. Therefore, since $supp(w)$ is complete we have that $supp(v)$ is also complete, so by the classification in Theorem \ref{decimo}, it follows that $v\in \zz(\D)$. But in all the cases this means that $\nmap(w)$ is a weak zigzag, so by Remark \ref{remark:weakzigBeD} it follows that $w\in \zz(\B)$, which is a contradiction.

    Assume $f_\bullet(\nmap(w))\ne 0$, we will show that $s_0,s_1\in \ld(\nmap(w))\cap \rd(\nmap(w))$. We can write $\nmap(w)=vv'$ reduced, with $\rd(v)=\{s_0,s_1\}$ and $f_\bullet(v)=1$. If $v\ne s_0s_1$, hence $v$ is a weak zigzag by Lemma \ref{lemma:condzigzag}. Moreover, since $w \in \IR(\B)$, $f_\bullet(v)=1$ and $s_{n+2}\notin \ld(\nmap(w))$, we have that $v=s_{n+1}s_n\cdots s_2s_1s_0$. We observe that $s_{n+2}\notin \ld(v')$, otherwise $s_{n+1}s_{n}s_{n+1}$ is a prefix of $w$. In particular, $s_{n+2}\nmap(w)\in \fc(\D)$ is a reduced product, so $s_{n+2}\nmap(w)$ admits $s_{n+2}v\in \zz(\D)$ as a prefix, thus $s_{n+2}\nmap(w)$ is a weak zigzag by Corollary \ref{corollary:precompzigzag}. Therefore, $\nmap(w)$ is a weak zigzag, and $w\in \zz(\B)$ by Remark \ref{remark:weakzigBeD}, but this is a contradiction. 
    So, $v=s_0s_1$ and $s_0,s_1\in \ld(\nmap(w))$. To prove that $s_0,s_1\in \rd(\nmap(w))$, we factorize $\nmap(w)=y'y$ with $\ld(y)=\{s_0,s_1\}$ and $f_\bullet(y)=1$ and we argue as above. Therefore, $s_0,s_1\in \ld(\nmap(w))\cap \rd(\nmap(w))$, thus $f_\bullet(\nmap(w))\ge 2$ since $s_2\in supp(\nmap(w))$. This implies that $\nmap(w)$ is reducible to $x\in \zz(\D)$ by Remark \ref{rem:occirr} and Theorem \ref{decimo}. Hence $\nmap(w)$ is a weak zigzag by Corollary \ref{corollary:precompzigzag}, in particular $\nmap(w)=x\in \zz(\D)$ since $s_0,s_1\in \ld(\nmap(w))\cap \rd(\nmap(w))$. But this is a contradiction because $w$ is not a complete zigzag. In conclusion, $f_\bullet(\nmap(w))=0$. 
    
    Now we prove that $f_\circ(\nmap(w))=0$, as above we factorize $\nmap(w)=v v'$ with $\rd(v)=\{s_{n+1},s_{n+2}\}$ and $f_\circ(v)=1$. If $v\ne s_{n+1}s_{n+2}$, then $v$ is a weak zigzag by Lemma \ref{lemma:condzigzag}. But this cannot be since $f_\bullet(\nmap(w))=0$, $f_\circ(v)=1$ and $w\in \IR(\B)$. 
    Therefore, $v=s_{n+1}s_{n+2}$ and $s_{n+1},s_{n+2}\in \ld(\nmap(w))$, but this is a contradiction by the definition of $\nmap$.

    Assume now $s_0,s_1\notin \ld(w)$ and let $\nmap(w)=u_0\cdots u_m$ be its CFNF. It follows that $s_0,s_1\notin supp(u_0)$, so by applying Lemma \ref{cornormalform} there exists $z\in \pr(\nmap(w))$ with \[z\in \su(s_{n+1
    } s_n\cdots s_2\s ) \] and $\ell(z)> 1$. We observe that in all the cases $z$ is a prefix of $w$ too, by the definition of the map. So the only possibility is that $z=s_{n+1} s_n\cdots s_2\s$ and we factorize $w$ as $w=zw'$ reduced product. We observe that $\ld(w')\subseteq \{\ti\}$ by $w\in \fc(\B)$ and $w\in \IR(\B)$. But if $\ti\in \ld(w')$, then $f_\bullet(\nmap(w))\ge1$, that contradict point (b). Therefore, $w'=e$ and $w=z$ but this a contradiction because $w\in \IR(\B)$. 
   
\end{proof}

\begin{Lemma}
\label{lemma:lemmaBdue}
    Let $w\in \IRc(\B) \setminus \zz(\B)$ with complete support, and let $w=u_0\cdots u_m$ be its CFNF. If $s_ns_{n+1}\in \pr(w)$, then the following hold:
    \begin{itemize}
        \item[(a)] $supp(u_0)=\{\s\} \cup \{s_3, s_5, \ldots, s_{n-2},s_n\}$;
        \item[(b)] $supp(u_1)=\{s_2, s_4, \ldots, s_{n-1},s_{n+1}\}$.
    \end{itemize}
    In particular, $n$ is odd. 
\end{Lemma}

\begin{proof}
    It suffices to show that, given $\nmap(w)=v_0\cdots v_m$ be its CFNF we have
    \begin{align*}
        supp(v_0)=\{\s\} \cup \{s_3, s_5, \ldots, s_{n-2}, s_n\},\quad\mbox{and} \quad
        supp(v_1)=\{s_2, s_4, \ldots, s_{n-1}, s_{n+1}\}. 
    \end{align*}
    We consider the following sets
    %This claim can be proved with the same argument that we use in Lemma \ref{cardue}. 
    \begin{align*}
    V_0:=\{2j+1 \mid s_{2j+1}\notin supp(v_0),  1 \leq 2j+1 \leq n-1\},\,\, V_1:=\{ 2j \mid s_{2j}\notin supp(v_1), 1\le 2j\le n-1\}. 
\end{align*}
We need to prove that $V_0$ and $V_1$ are both empty. Following verbatim the steps of the proof of Lemma \ref{cardue}, we obtain that there exists $z\in \pr(w)$ with
\[z\in \su(\s s_2\cdots s_{h-1}) \quad \mbox{or} \quad z\in \su(s_{n+1} s_n\cdots s_{h+1})\] and $h\ge 2$ and $\ell(z)\ge 2$. Both cases leads to a contradiction since $w\in \IRc(\B)$ and $s_{n}s_{n+1}\in \pr(w)$. In particular, it follows that $n$ is odd, otherwise $s_{n-1},s_{n}\in \ld(w)$. 

\end{proof}

\begin{Proposition}
\label{proposition:weakcara}
Let $n$ be odd and $w\in \IR(\B) \setminus \zz(\B)$ with complete support. Then $w\in \carw(\B)$. 
\end{Proposition}

\begin{proof}
We observe that $w\in \IRc(\B)$, otherwise $w\in \zz_\bullet(\B)$ by Theorem \ref{theorem:classredB}. Without loss of generality, by Lemma \ref{lemma:lemmaBuno}, assume $s_{n}s_{n+1}\in \pr(w)$. Let $w=u_0\cdots u_m$ be its CFNF, then we know that $supp(u_0)$ and $supp(u_1)$ are described in Lemma \ref{lemma:lemmaBdue} (a) and (b). This implies $m\ge 2$, otherwise $\s s_2$ would be a prefix of $w$, which is not possible since $w\in \IR(B)$. Since $w\in \fc(\B)$, by definition of CFNF, we have that
\begin{align}
    supp(u_2)\subseteq \{\ti\} \cup \{s_3,s_5,\ldots, s_n\}:=\mathcal{S}. \label{eq:incl}
\end{align}
Define $w_2:=u_2\cdots u_m$ and note that is in $\IR(\B)$. In fact, it is right weak reducible since it is a suffix of $w\in \IR(\B)$. Moreover, it is also left weak irreducible because if $st$ with $m_{s,t}=3$ is a prefix of $w_2$, then $s\in supp(u_2)$ and $t\in supp(u_3)$, but by CFNF we have that $supp(u_3)\subseteq supp(u_1)$, so $tst$ would be a factor of $w$, that contradicts the hypothesis of fully commutativity. On the other hand, if $sts$ with $m_{s,t}=4$ is a prefix of $w_2$, then necessarily $s=s_{n}$ and $t=s_{n+1}$, but in this way we have that $s_{n+1}s_ns_{n+1}s_n$ is a factor of $w$. 

To summarize, we have $w=u_0u_1w_2$, with $w_2=u_2\cdots u_m\in \IR(\B)$ and $m\ge 2$. Now we show by induction on $m\ge 2$ that $w\in \carw(\B)$.

Suppose $m=2$, then the inclusion in \eqref{eq:incl} is an equality. Otherwise there would exist $s'\in \mathcal{S}\setminus supp(u_2)$, $s\in supp(u_1)$, $t \in supp(u_2)$ with $m_{s,t}=m_{s',s}=3$, and the suffix $st$ would appear in $w$, that is not possible. Hence $w=u_0u_1u_2\in \carw(\B)$; its heap is depicted in Figure \ref{fig:small_candyB}.

\begin{figure}[h]
    \centering
    \includegraphics[width=0.5\linewidth]{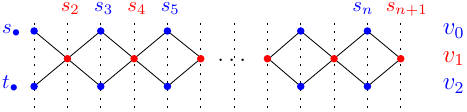}
    \caption{Heap of $u_0u_1u_2$ in the proof of Proposition \ref{proposition:weakcara}.}
    \label{fig:small_candyB}
\end{figure}

If $m>2$, we first show that $supp(w_2)$ is complete. In fact, if $supp(w_2)$ is not complete, then $w_2$ is of finite type $D$, $B$ or a product of elements of type $D$ and $B$. Hence $w_2\in \TTCw(\B)$ by Remark \ref{rem:stared} and \cite[Theorem 4.2.1]{ernst} because $w_2\in \IR(\B)$. Since $m>2$ and $s_{n+1}\in supp(u_1)$, the unique possibility is that $w_2=s_{n}s_{n+1} w'=w's_{n}s_{n+1}$ reduced product with $w'\in \TTC(\B)$. But in this case, since $s_{n+1}w'=w's_{n+1}$ and $s_{n+1}\in supp(u_1)$, we have that $s_{n+1}s_ns_{n+1}\in \su(w)$ which contradicts the weak irreducibility.

Therefore, $supp(w_2)$ is complete and $w_2\in \su(w)$, so $w_2$ is not a complete zigzag since $f_\bullet(\nmap(w))=f_\bullet(\nmap(w_2))=0$ by Lemma \ref{lemma:lemmaBuno}. So, by inductive hypothesis, $w_2$ is a weak candy element, hence $w$ is a weak candy element as well, since $w=u_0u_1w_2$. 

\end{proof}

\begin{Theorem}[Classification of irreducible elements of type $\B$]
\label{undicesimo}
\noindent
The set of irreducible elements in $\fc(\B)$ is partitioned as follows:
\begin{align*}
    \IR(\B) = \car (\B) \sqcup \carw(\B)\sqcup \zz(\B) \sqcup \TTCw(\B).
\end{align*}
If $n$ is odd, $\car(\B)=\emptyset$; if $n$ is even, $\carw(\B)=\emptyset$.
\end{Theorem}

\begin{proof}
Let $w\in \IR(\B)$, if its support is not complete then it is of type $D$, $B$ or a product of elements of type $D$ and $B$. 
By Remark \ref{rem:stared} and \cite[Theorem 4.2.1]{ernst}, we have that $w\in \TTCw(\B)$. Assume that $supp(w)$ is complete and $w$ is not a complete zigzag; 

if $n$ is odd, then $w\in \carw(\B)$, by Proposition \ref{proposition:weakcara}.
Otherwise, if $n$ is even by Lemmas \ref{lemma:lemmaBuno} and \ref{lemma:lemmaBdue}, we have that $w\in \I(\B)$ so $w\in \car(\D)$ by Theorem \ref{theorem:classredB} since $w\notin \zz(\B)$. 
%%$\nmap(w)\in \I(\D)$ and $\nmap(w)\notin \zz(\D)$. Therefore, $\nmap(w)\in \car(\D)$ by Proposition \ref{cartre}, hence $w\in \car(\B)$.
\end{proof}

%%%%%%%%%%%%%%%%%%%%%%%%%%%%%%%%%%%%
\section{Decorated diagrams}\label{sec:diagrams}
%%%%%%%%%%%%%%%%%%%%%%%%%%%%%%%%%%%%

%%%%%%%%%%%%%%%%%%%%%%%%%%%%%%%%%%%%%%%
\subsection{Temperley--Lieb diagrams}
%%%%%%%%%%%%%%%%%%%%%%%%%%%%%%%%%%%%%
The following paragraph provides a brief overview of the classical Temperley--Lieb diagrams. For a more comprehensive treatment, we refer the reader to \cite[\S 3]{ErnstDiagramI} and the references therein. A \textit{pseudo k-diagram} consists of a finite number of disjoint plane curves, called \textit{edges}, embedded in the standard $k$-box, that is a rectangle with $2k$ marked points called \textit{nodes} or \textit{vertices}; we will refer to the top of the rectangle as the \textit{north face} and to the bottom as the \textit{south face}. The nodes are endpoints of edges. All other embedded edges must be closed (isotopic to circles) and disjoint from the box. We refer to a closed edge as a \textit{loop}. It follows that there cannot exist isolated nodes and that a single edge starts from each node (an example is given in Figure \ref{7box}). A \textit{non-propagating edge} is an edge that joins two nodes of the same face, while a \textit{propagating edge} is an edge that joins a node of the north face to a node of the south face. If an edge joins the nodes $i$ and $i'$ we call it a \textit{vertical edge}. We denote by $\mathbf{a}(D)$ the number of non-propagating edges on the north face of a concrete pseudo $k$-diagram $D$. The pseudo $k$-diagrams are defined up to isotopy equivalence; denote the set of the pseudo $k$-diagrams by $T_k(\emptyset)$.

\begin{figure}[hbtp]
\centering
\includegraphics[scale=0.4]{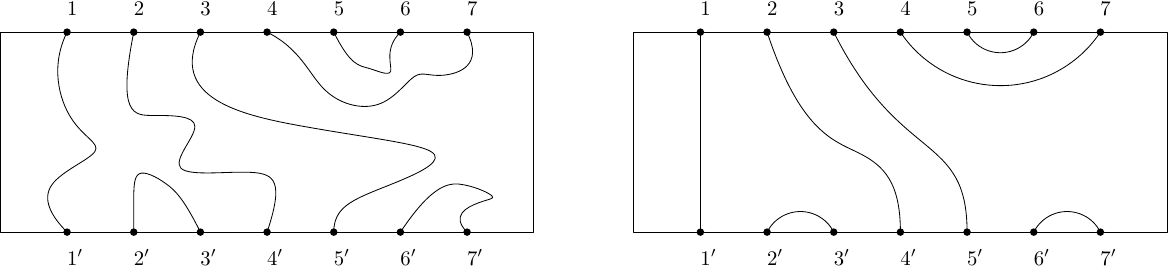}
\caption{Two equivalent pseudo 7-diagrams.}
\label{7box}
\end{figure}

In this work, we are interested in a decorated generalization of the family of classical Temperley--Lieb diagrams, which we refer to as \textit{$LR$-decorated diagrams of type $\widetilde{D}$}. 
We begin by recalling the basic notions required for their definition, referring to \cite[\S 3]{BFS}, where these diagrams were first introduced, for a detailed treatment.
Briefly, an $LR$-decorated pseudo $k$-diagram is a pseudo $k$-diagram $D$ whose edges are equipped with finite sequences of \textit{decorations} from the set $\Omega=\{\bullet, \circ\}$, subject to a collection of constraints. These rules, labeled (D0)–(D4) in \cite[\S 2]{BFS}, specify the admissible configurations of decorations. We denote the set of $LR$-decorated pseudo $k$-diagrams by $\TLR$.
If an edge \textit{e} is decorated, we read off the sequence of decorations left to right if $e$ is a non-propagating edge, up to down if $e$ is a propagating edge.
For a loop edge $e$, the sequence is read according to any arbitrarily chosen starting point and any direction around the loop. 

\begin{Definition}
 Let $D$ be a pseudo $k$-diagram with $\ab(D)=1$, decorated with at least one decoration in $\Omega$ and such that the unique non-propagating edge and loops (if any) have at most one type of decoration. A $L$-\textit{strip} (respectively $R$-\textit{strip}) is a part of the diagram delimited by two horizontal lines that contains at least one $\bullet$-decoration (respectively $\circ$-decoration) and no $\circ$-decorations (respectively $\bullet$-decorations). %A \textit{strip partition} of $D$ is a partition into \textit{strips} in which $L$-strips and $R$-strips alternate. 
\end{Definition}

\begin{Definition}
We define $\PLR$ to be the free $\Zd$-module having the elements of $\TLR$ as a basis.
\end{Definition}

Now we define the product of two elements of $\TLR$ as the concatenation, and then extend this bilinearly to define a multiplication in $\PLR$. To concatenate two diagrams $D,D'\in \TLR$, we place $D'$ on top of $D$ so that node $i$ of $D$ coincides with node $i'$ of $D'$, conjoin adjacent blocks % maintaining $\Omega$-equivalence 
and then rescale vertically by a factor of $1/2$. 

\begin{figure}[hbtp]
\centering
\includegraphics[scale=0.6]{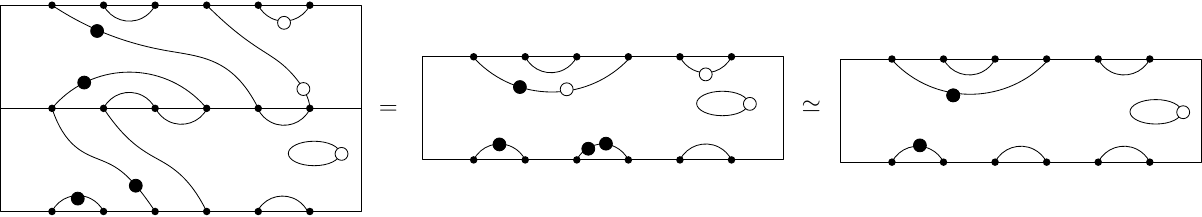}
\caption{Product in $\mathcal{P}^{LR}_6$ and reduction in $\widehat{\mathcal{P}}^{LR}_6$.}
\label{prodec}
\end{figure}

The next result is proved in \cite[Proposition 2.5]{BFS}.

\begin{Proposition}\label{decoralgebra}
The module $\PLR$ with the aforementioned product is a well-defined associative $\Zd$-algebra, with identity element the pseudo $k$-diagram $I_k$ made of $k$ undecorated vertical edges. 
\end{Proposition}

Note that $\PLR$ is an infinite dimensional algebra since, for instance, when $D$ has a single propagating edge, there is no limit to the number of decorations on such edge, or when $D$ has no propagating edges, there is no limit to the number of loops with both decorations. 

We introduce the notation for three types of loop edges in Figure \ref{loop-D}.

 \begin{figure}[h]
	\centering
	\includegraphics[width=0.45\linewidth]{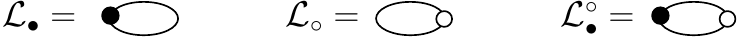}
	\caption{Special types of loops.}
	\label{loop-D}
 \end{figure}

We define a \textit{reduction system} as depicted in Figure \ref{rel}, where the lines represent a portion of an edge (loops included). 

\begin{figure}[hbtp]
\centering
\includegraphics[scale=0.6]{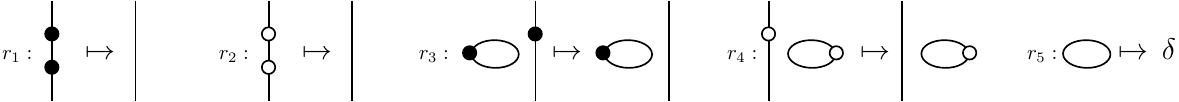}
\caption{The reduction system in $\PLR$.}
\label{rel}
\end{figure}

More precisely the reductions $r_i$ described in Figure \ref{rel} are defined as follows:
\begin{enumerate}   
    \item $r_1$ and $r_2$ reduce the number of decorations and they are applied to adjacent decorations in the same block;
     \item if $\bO$ (respectively, $\wO$) occurs in $D$, then $r_3$ (respectively, $r_4$) removes a $\bullet$ (respectively, $\circ$) on another edge (loop edge included) with the only restriction that if $\ab(D)=1$ then $\bO$ (respectively, $\wO$) and $\bullet$ (respectively, $\circ$)  must belong to the same strip. 
     \item $r_5$ removes an undecorated loop edge multiplying $D$ by a factor $\delta$.
\end{enumerate} 
Clearly, if $\ab(D)>1$, a finite sequence of applications of $r_3$ (respectively, $r_4$) removes all $\bullet$ (respectively, $\circ$) in $D$ except the one on the last remaining loop, while in the case $\ab(D)=1$, it only removes all $\bullet$ (respectively, $\circ$) that belong to the same $L$-strip (respectively, $R$-strip).
\smallskip

Let $\PLRk$ be the quotient of $\PLR$ modulo the two-sided ideal generated by the relations determined by the reductions in Figure \ref{rel}. 

\smallskip
Now define the \textit{simple diagrams} $D_0, \ldots, D_{n+2}$ as in Figure \ref{simple}. Since they are irreducible, they are basis elements of $\PLRn$.  
\begin{figure}[hbtp]
\centering
\includegraphics[scale=0.6]{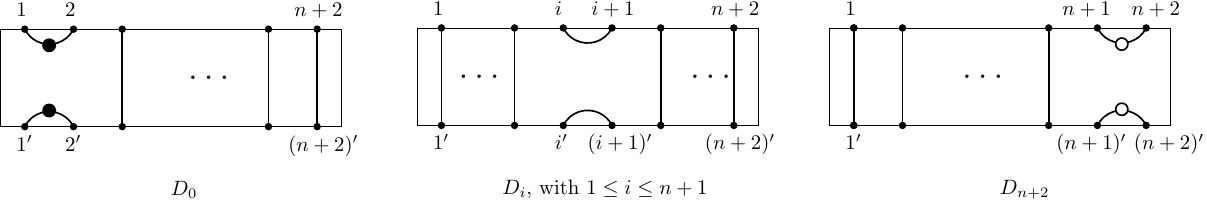}
\caption{The simple diagrams.}
\label{simple}
\end{figure}
It is easy to prove that the simple diagrams satisfy the relations (d1)-(d3)  listed in Definition~\ref{def:tl-algebras}, where $b_i$'s are replaced by $D_i$'s.
Moreover, we call \textit{simple edge} a non-propagating edge appearing in a simple diagram. The simple edges are denoted by 
\begin{equation}
\label{eq:edgenotation}
    \{1,2\}_{\bullet}, \ \{1',2'\}_{\bullet}, \ \{n+1,n+2\}_{\circ},\ \{(n+1)',(n+2)'\}_{\circ}, \ \{i,i+1\}, \ \{i',(i+1)'\},
\end{equation} 
where $1\le i\le n+1$ and the possible decorations are underscored. 

\begin{Definition}
Let $\Di(\D)$ be the $\Zd$-subalgebra of $\PLRn$ generated as a unital algebra by the simple diagrams with multiplication inherited by $\PLRn$.
\end{Definition}

We define the $\Zd$-algebra homomorphism 
\begin{eqnarray}
    \tilde{\theta}_D: \tl(\D) & \longrightarrow & \Di(\D) \\
    b_i &\mapsto & D_i \nonumber
\end{eqnarray} 
for all $i=0,\ldots,n+2$, where $\tl(\D)$ is the $\Zd$-algebra defined in Definition \ref{def:tl-algebras}. If $w=s_{i_1}\cdots s_{i_k}\in \fc(\D)$, then $\tilde{\theta}_D(b_w)=D_w:=D_{i_1}\cdots D_{i_k}$ is well-defined because the relations (d1)-(d3) listed in Definition \ref{def:tl-algebras}, are satisfied by the simple diagrams. By \cite[Theorem 5.14]{BFS}, $\tilde{\theta}_D$ is an algebra isomorphism.
In the present work, we provide an alternative algebraic proof of this result, based on the structure of the irreducible elements of type~$\D$.

%%%%%%%%%%%%%%%%%%%%%%%%%%%%%%%%%
\subsection{Irreducible diagrams}
%%%%%%%%%%%%%%%%%%%%%%%%%%%%%%%%%
\begin{Lemma}
\label{discesadelta}
Let $w\in \fc(\D)$, if $s\in \ld(w)$ (resp. $s\in \rd(w)$) then $D_sD_w=\delta D_w$ (resp. $D_wD_s=\delta D_w$).
\end{Lemma}

\begin{proof}
If $s\in \ld(w)$, then $w=su$, $u\in \fc(\D)$ and $\ell(w)=\ell(u)+1$. Therefore, $D_w=D_sD_u$, so it follows that \[D_sD_w=D_sD_sD_u=\delta D_sD_u=\delta D_w.\] The same argument can be applied to $s\in \rd(w)$. 
\end{proof}

\begin{Lemma}
\label{discesearchi}
Let $w \in \fc(\D)$. 
\begin{enumerate}
    \item[(a)] For $1\le i\le n+1$, if $s_i \in \ld(w)$, then $\{i,i+1\}$ appears in $D_w$.
    \item[(b)] If $s_0 \in \ld(w)$, then either $\{1,2\}_{\bullet}$ appears in $D_w$, or $\{1,2\}$ appears in $D_w$ and the numbers of occurrences of $\bO$ in $D_w$ and $D_{0}D_{w}$ are the same. 
    \item[(c)] If $s_{n+2} \in \ld(w)$, then either $\{n+1,n+2\}_{\circ}$ appears in $D_w$, or $\{n+1,n+2\}$ appears in $D_w$ and the numbers of occurrences of $\wO$ in $D_w$ and $D_{n+2}D_{w}$ are the same.
\end{enumerate}
\end{Lemma}

\begin{proof}
Note that $(a)$ follows directly by the definition of diagram concatenation. 

If $w=s_0u$ is a reduced product, before applying any reductions (if any), the edge $\{1,2\}_\bullet$ appears in the concatenation of $D_0$ with $D_u$. If $D_w$ does not contain a $\bO$ such that reduction (r3) is applicable, then the thesis follows. Otherwise, after applying (r3) to the edge $\{1,2\}_\bullet$, it looses its decoration, and $\{1,2\}$ is in $D_w$. Moreover, the same occurrence of $\bO$ gives rise to a reduction (r3) of the new occurrence of $\bO$, obtained from the concatenation of $D_0$ with $D_w$, to an undecorated loop. Note that if $\mathbf{a}(D_w)>1$, then $\# \bO=1$. Hence the numbers of $\bO$ in $D_w$ and $D_{0}D_{w}$ are the same. This proves (b); (c) is analogous.
\end{proof}

A direct consequence of Lemma $\ref{discesearchi}$ is that $D_w=I_{n+2}$ if and only if $w=e$.

\begin{Proposition}
\label{diagrammiasscomm}
Let $w\in \TTC(\D)$, then:
\begin{enumerate}
    \item[(a)] the only types of edges allowed in $D_w$ are simple edges and vertical edges $\{j,j'\}$ with $1\le j \le n+2$;
    \item[(b)] the only types of loops allowed in $D_w$ are $\bO$ and $\wO$;
    \item[(c)] $f_\bullet(w)=\# \bO\le 1$ and $f_\circ(w)=\# \wO \le 1$ in $D_{w}$.
\end{enumerate}
\end{Proposition}

\begin{proof}
Let $w = s_{i_1} \cdots s_{i_k}$ reduced with $ m_{s_{i_j}, s_{i_l}} = 2 $ for all $ j \ne l $, then $D_w = D_{i_1} \cdots D_{i_k}$. Observe that the simple edges associated with two distinct commuting generators, if they share a vertex, then they share both vertices, and the generators must be either $ s_0 $ and $ s_1 $, or $ s_{n+1} $ and $ s_{n+2} $. For this reason, the only allowed edges in $ D_w $ are either simple non-propagating edges, or edges of the form $ \{j, j'\} $ with $ 1 \le j \le n+2 $, and the only allowed loops are those of the form $ \bO $ and $ \wO $.

Moreover, it can be observed that a loop $ \bO $ is present in $ D_w $ if and only if $ s_0, s_1 \in supp(w) $, due to the earlier observation on possible shared nodes between simple edges associated with commuting generators. Similarly, the same holds for $ \wO $ and the occurrences of $ s_{n+1} $ and $ s_{n+2} $. Therefore, we conclude that $f_\bullet(w) = \# \bO$ and $f_\circ(w) = \# \wO$ in $ D_w $.

\end{proof}

\begin{figure}[h!]
    \centering
    \includegraphics[width=0.4\linewidth]{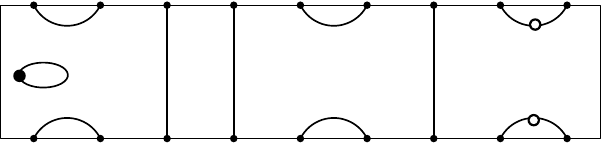}
    \caption{$D_w$ with $w=s_0s_1s_5s_9$ in $\TTC(\widetilde{D}_9)$.}
    \label{excomm}
\end{figure}

\begin{Proposition}
\label{diagrammiasscaramelle}
Let $n$ be even and $w=u_0\cdots u_m\in \car(\D)$. Then any loop in $D_w$ is of the form $\bwO$, and the number of occurrences of $\bwO$ is $m/2$.
\end{Proposition}

\begin{proof}
Recall that $m$ is even since $w\in \car(\D)$.
Without loss of generality assume $s_0,s_{n+2}\in \ld(w)$ since the other cases are analogous. By definition of candy elements we have that $u_0= s_0s_3\cdots s_{n-1}s_{n+2}$ and for $1\leq i \le m-1$,
\[u_iu_{i+1}=
\begin{cases}
    s_2s_4\cdots s_{n-2}s_ns_1s_3\cdots s_{n-1}s_{n+1}, \mbox{ for $i$ odd}; \\
    s_2s_4\cdots s_{n-2}s_ns_0s_3\cdots s_{n-1}s_{n+2}, \mbox{ for $i$ even}.
\end{cases}\]
Then $D_w=D_{u_0}\cdots D_{u_i u_{i+1}}\cdots D_{u_{m-1}u_m}$, and each concatenation gives rise to a loop $\bwO$, so $\# \bwO=m/2$; see Figure $\ref{concdiagcara}$, left. No other types of loops are formed through the concatenation. 
\end{proof}

\begin{Proposition}
\label{diagrammiasszigzag}
Let $w \in \zz(\D)$, then $\mathbf{a}(D_w)=1$ and any loop in $D_w$ is of the forms $\bO$ or $\wO$, moreover $f_\bullet(w)=\# \bO$ and $f_\circ(w)=\# \wO$.
\end{Proposition}

\begin{proof}
By Definition \ref{def:zigzag}, either $D_w=D_0D_1(D_AD_B)^kD_A^h$ or $D_w=D_{n+1}D_{n+2}(D_BD_A)^kD_B^h$ with $k\ge 0$, $h\in\{0,1 \}$ and $k+h>0$. A direct computation shows that $\mathbf{a}(D_w)=1$ and that each concatenation by $D_A$ gives rise to a $\wO$ and each concatenation by $D_B$ to a $\bO$. Hence, in the first case, the number of occurrences of $\bO$ (respectively, $\wO$) in $D_w$ is $k+1$ (respectively, $k+h$) which is equal to $f_\bullet(w)$ (respectively, $f_\circ(w)$). A similar computation holds for the second case. No other types of loops are formed through these concatenations. An example is given in Figure \ref{concdiagcara}, right.

\end{proof}

\begin{figure}[ht]
    \centering
    \rotatebox{90}{\includegraphics[width=0.35\linewidth]{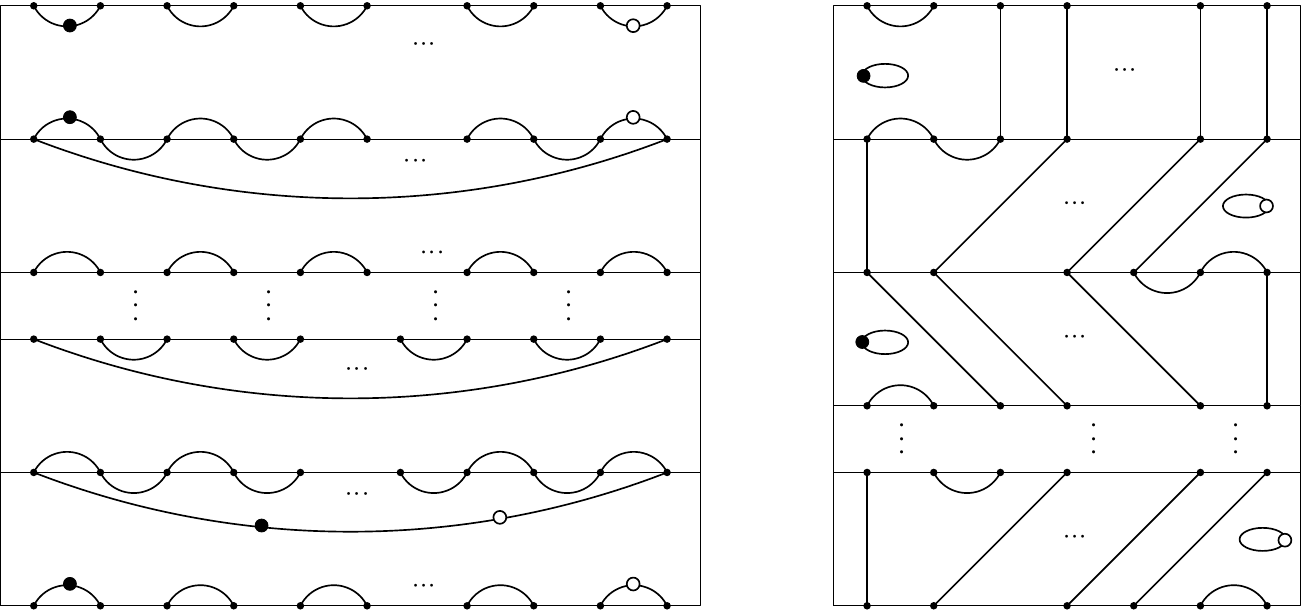}}
    \caption{On the left, the concatenation $D_w=D_{u_0}D_{u_1u_2}\cdots D_{u_{m-1}u_{m}}$ in the proof of Proposition \ref{diagrammiasscaramelle}; on the right, the concatenation $D_w=D_0D_1(D_AD_B)^kD_A$ in the proof of Proposition \ref{diagrammiasszigzag}.}
    \label{concdiagcara}
\end{figure}

\begin{Proposition}
\label{lemmadelta}
Let $w \in \fc(\D)$. Then 
\begin{itemize}
    \item[(a)] the diagram $D_w$ does not contain any undecorated loops;
    \item[(b)] $f_\bullet(w) = \# \bO$ in $D_w$;
    \item[(c)] $f_\circ(w) = \# \wO$ in $D_w$. 
\end{itemize}
\end{Proposition}

\begin{proof}
We show that the number of undecorated loops, $\bO$ and $\wO$ are invariant by left and right reductions. The result then follows since (a), (b), and (c) hold for irreducible elements by Propositions $\ref{diagrammiasscomm}$, $\ref{diagrammiasscaramelle}$ and $\ref{diagrammiasszigzag}$. 

Without loss of generality, assume that $w\red_t sw=:w_1$. Then \[  D_t D_w=D_tD_sD_{w_1}=D_tD_sD_tD_u=D_tD_u=D_{w_1} \] where $w_1=tu$ reduced. Since $s\in \ld(w)$, then the simple edge associated to $s$ by Proposition \ref{discesearchi} appears in $D_w$. Thus, the concatenation $D_t D_w$ does not create new loops, decorated or not, since the simple edge in $D_t$ is adjacent to the one corresponding to $s$ in $D_w$ (see Figure \ref{figloop}). Hence, the total number of loops does not change, and (a) is proved. Moreover, by Remark \ref{rem:occirr}(2), $f_\bullet(w)=f_\bullet (w_1)$ and $f_\circ(w)=f_\circ(w_1)$, so (b) and (c) also yield. 
\end{proof}

\begin{figure}[h!]
    \centering
    \includegraphics[width=0.35\linewidth]{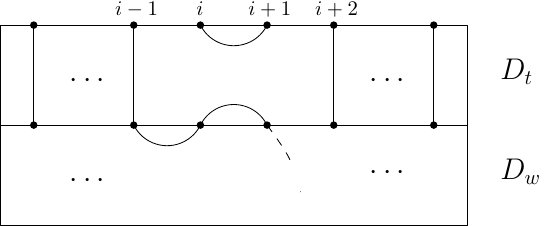}
    \caption{$D_tD_w=D_{w_1}$, for Proposition $\ref{lemmadelta}$.}
    \label{figloop}
\end{figure}

%%%%%%%%%%%%%%%%%%%%%%%%%%%%%%%%%%%%%%%
\section{Descents and monomial basis}\label{sec:descents}
%%%%%%%%%%%%%%%%%%%%%%%%%%%%%%%%%%%%%%%
In Definition \ref{def:tl-algebras}, we introduced $\tl(\D)$ generated by $\{b_0, b_1, \ldots, b_{n+2}\}$ and defining relations (d1)-(d3). Recall that $\{b_w\mid w\in \fc(\D)\}$ is a basis for $\tl(\D)$, called monomial basis.

\begin{Remark}
\label{remark:tlalgebra}
It is not difficult to see that by applying the relations (d1)-(d3), $b_{{i_1}}\cdots b_{{i_r}}=\delta^k b_x$ with $k\ge 0$ and $x\in \fc(\D)$. In particular, $s_{i_1}\in \ld(x)$ and $s_{i_r}\in \rd(x)$.
% forse meglio metterlo come un semplice remark dato che deriva strettamente della presentazione dell'algebra tl. 
\end{Remark}

\begin{Lemma}
\label{nuovoprel}
Let $w\in \fc(\D)$ and suppose there exist $s\in S$ and $s'\in \ld(w)$ with $m_{s,s'}=3$. Then $b_sb_w=b_x$ with $x\in \fc(\D)$, $f_\bullet(w)=f_\bullet(x)$ and $f_\circ(w)=f_\circ(x)$.
\end{Lemma}

\begin{proof}
First of all observe that $\ell(sw)= \ell(w)+1$. If $sw\in \fc(\D)$ then $b_sb_w=b_{sw}$ and it follows that $f_\bullet(w)=f_\bullet(sw)$ and $f_\circ(w)=f_\circ(sw)$. If $sw\notin \fc(\D)$, then $w$ can be written as $w=s'su$ reduced with $m_{s',s}=3$ by Remark \ref{oss}. Hence, $b_sb_w=b_{s}b_u=b_{su}$ with $f_\bullet(w)=f_\bullet(su)$ and $f_\circ(w)=f_\circ(su)$ by Remark $\ref{rem:occirr}$.
\end{proof}

\begin{Proposition}
\label{nuovodue}
Let $w\in \fc(\D)$ be reducible to $v\in \car(\D)\cup \TTC(\D)$, and let $s\notin \ld(w)$ be such that $b_sb_w=\delta ^kb_{x}$ with $x\in \fc(\D)$. 
\begin{itemize}
    \item[(a)] If $k>0$, then $v\in \car(\D)$ and $f_\bullet(x)=f_\circ(x)=1$.
    \item[(b)] If $v\in \TTC(\D)$, then $k=0$.
\end{itemize}
%If $s\notin \ld(w)$, $b_sb_w=\delta ^kb_{x}$ with $k>0$ and , then  $f_\bullet(x)=f_\circ(x)=1$ and $v\in \car(\D)$.
\end{Proposition}

\begin{proof}

Note that (b) follows by (a). 
To prove (a) we consider a reduction sequence $w_0=w\red w_1\red w_2\red \cdots \red w_l=v$ with $v=u_0\cdots u_m$ its CFNF.
We prove the result by induction on $l$.

If $l=0$, then $w=v\in \car(\D)\cup \TTC(\D)$. If $w\in \TTC(\D)$, then $sw\in \fc(\D)$ and $b_sb_{w}=b_{sw}$. So, we have $w\in \car(\D)$ and we assume $\{\s , \sn\}\subset \ld(w)$. If $s=s_j$ with $2\le j\le n$ even, then there exists $s'\in \ld(w)$ such that $m_{s,s'}=3$, so $b_sb_w=b_x$ by Lemma $\ref{nuovoprel}$. Thus, $s\in \{\ti, \tn\}$, then $b_sb_w=\delta ^kb_x$ with $k=\lfloor m/2\rfloor -1$ and $x=s_0s_1s_3s_5\cdots s_{n-1}s_{n+1}s_{n+2}\in \TTC(\D)$ by direct computation and $f_\bullet(x)=f_\circ(x)=1$. 

Assume $l>0$. First, we assume that $w\red_p rw=:w_1$. We have that $m_{s,r}=2$, otherwise if $m_{s,r}=3$, then $b_sb_w=b_y$ with $y\in \fc(\D)$ by Lemma \ref{nuovoprel}. Therefore, 
\[\delta ^kb_{x}=b_sb_w=b_sb_{r}b_{w_1}=b_{r}b_{s}b_{w_1}\] with $k>0$ and $s\notin \ld(w_1)$.
Now we distinguish two cases. Suppose $m_{s,p}=2$. By Remark \ref{remark:tlalgebra} we have that $b_sb_{w_1}=\delta ^hb_{y}$ with $h\ge 0$, $y\in \fc(\D)$ and $p\in \ld(y)$, so \[b_rb_{s}b_{w_1}=\delta ^hb_{r}b_{y}.\]
Thus by Lemma \ref{nuovoprel}, $b_rb_y=b_{x'}$ with $x'\in \fc(\D)$ and $f_\bullet(x')=f_\bullet(y)$ and $f_\circ(x')=f_\circ(y)$. Therefore, $\delta ^kb_{x}=\delta ^hb_{x'}$, hence $h=k>0$ and $x=x'$. By inductive hypothesis on $w_1$, we have that $f_\bullet(y)=f_\circ(y)=1$ and $v\in \car(\D)$, in particular \[f_\bullet(x)=f_\bullet(y)=1 \mbox{ and }f_\circ(x)=f_\circ(y)=1.\]

On the other hand, if $m_{s,p}=3$, we have that $sw_1\notin \fc(\D)$, otherwise if $sw_1\in \fc(\D)$, then $rsw_1=sw\in \fc(\D)$ and $b_sb_w=b_{sw}$. 
Thus, by Remark \ref{oss}(5), it means that $w_1\red_s pw_1:=w'$. Therefore, $w$ is star reducible to $w'$, in particular by Theorem \ref{theorem:starope}(2) there exists a sequence of length $l$ such that\[ w\red w_1'=w_1\red w_2'=w'\red \cdots \red w'_l=v'\in \car(\D)\cup \TTC(\D)\] 
by Corollary \ref{corollary:corclassification}. Thus, \[\delta ^kb_x=b_{r}b_{s}b_{w_1}=b_{r}b_{s}b_{p}b_{w'}=b_{r}b_{w'}\] with $k>0$, since $b_{s}b_{p}b_{w'}=b_{w'}$ and $s\in \ld(w')$. Moreover, since $w=rpw'$ reduced then $r\notin \ld(w')$. Therefore, by inductive hypothesis $f_\bullet(x)=f_\circ(x)=1$ and $v'\in \car(\D)$, so we have that $v=v'$ by Corollary \ref{corollary:corclassification}.

 Now, we suppose that $w\red_p wr:=w_1$. Thus, \[\delta ^kb_x=b_sb_w=b_sb_{w_1}b_{r}\] with $k>0$. By Remark \ref{remark:tlalgebra}, $b_sb_{w_1}= \delta ^hb_y$ and since $p\in \rd(w_1)$, we have that $p\in \rd(y)$. Thus, $b_yb_r=b_{x'}$ with $x'\in \fc(\D)$ and $f_\bullet(x')=f_\bullet(y)$ and $f_\circ(x')=f_\circ(y)$ by Lemma \ref{nuovoprel}. Therefore, $\delta ^k b_x=\delta^h b_{x'}$, so $k=h>0$ and $x=x'$. By inductive hypothesis on $w_1$, we have that $f_\bullet(y)=f_\circ(y)=1$ and $v\in \car(\D)$, in particular \[f_\bullet(x)=f_\bullet(y)=1 \mbox{ and }f_\circ(x)=f_\circ(y)=1.\]

\end{proof}

\begin{Remark}
\label{remark:a-value}
    Note that the $\mathbf{a}$-value is invariant for star reducibility. In fact, assume that $w\red_{t}sw=:w'$. Since $m_{s,t}=3$, the non-propagating edge on the north face associated to $t$ in $D_{w'}$ has a common node with the simple edge associated to $s$. Thus, when concatenating $D_s$ and $D_{w'}$, the number of non-propagating edges does not change. Therefore, since $D_w=D_sD_{w'}=D_{sw'}$, we have $\ab(D_w)=\ab(D_{w'})$.
\end{Remark}

\begin{Proposition}
\label{nuovotre}
Let $w\in \fc(\D)$ to be reducible to $v\in \zz(\D)$, then $\mathbf{a}(D_w)=1$. Moreover, if $s\notin \ld(w)$ and $b_sb_w=\delta ^kb_x$ with $x\in \fc(\D)$ and $k>0$, then $\mathbf{a}(D_x)\ge 2$. 
\end{Proposition}

\begin{proof}
Since $\ab(D_v)=1$ by Proposition \ref{diagrammiasszigzag}, we have that $\ab(D_w)=1$ by Remark \ref{remark:a-value}. 

To conclude the proof, let $s\notin \ld(w)$, $b_sb_w=\delta^k b_x$ with $k>0$, and $x\in \fc(\D)$. We have that $m_{s,t}=2$ for all $t\in \ld(w)$, otherwise if there exists $s'\in \ld(w)$ such that $m_{s,s'}=3$, then $b_sb_{w}=b_y$ with $y\in \fc(\D)$ by Lemma \ref{nuovoprel}. We observe that $\ld(x)\supseteq \ld(w)\sqcup \{s\}$ by Remark \ref{remark:tlalgebra}, so for this reason $\left | \ld(x)\right |\ge 2$ and moreover $\ld(x)\notin \{\{s_0,s_1\},\{s_{n+1},s_{n+2}\}\}$. In fact, for instance, if $\ld(x)=\{s_0,s_1\}$, then $s=\s$ and $w$ is a weak zigzag with $\ld(w)=\{\ti\}$, by Corollary \ref{corollary:precompzigzag}. But this means that $sw\in \fc(\D)$ and $b_sb_w=b_{sw}$. 
Therefore, by Lemma $\ref{discesearchi}$ we have that $\mathbf{a}(D_x)\ge 2$.

\end{proof}

\begin{Theorem}
\label{maindis}
Let $w\in \fc(\D)$, $s\in \ld(w)$ if and only if $D_sD_w=\delta D_w$. Therefore, if $D_w=D_u$, then $\ld(w)=\ld(u)$ and $\rd(w)=\rd(u)$. 
\end{Theorem} 

\begin{proof}
Assume $s\in \ld(w)$, then $D_sD_w=\delta D_w$ by Lemma \ref{discesadelta}. 

Suppose now, $D_sD_w=\delta D_w$ and $s\notin \ld(w)$. We observe that by Lemma \ref{nuovodue} and definition of $\widetilde{\theta}_D$, $w$ is reducible to $v\in \car(\D)\cup \zz(\D)$. By Remark \ref{remark:tlalgebra}, we have that $b_sb_w=\delta ^kb_x$ with $k\ge 0$ and $x\in \fc(\D)$. Hence, applying the map $\widetilde{\theta}_D$ we have that $\delta ^kD_x=\delta D_w$. Moreover, by Proposition \ref{lemmadelta}, it means that $k=1$ and $D_x=D_w$. Now, if $v\in \car(\D)$, then $0=f_\bullet(w)\ne f_\bullet(x)=1$ by Remark \ref{rem:occirr} and Proposition \ref{nuovodue}, but they should be the same by Proposition \ref{lemmadelta}. On the other hand, if $v\in \zz(\D)$, by Lemma \ref{nuovotre} we have that $2\le \mathbf{a}(D_x)\ne\mathbf{a}(D_w)=1$, which is absurd. Thus, $s\in \ld(w)$. 

In conclusion, if $D_w=D_u$, then for all $s\in S$, $D_sD_w=\delta D_w$ if and only if $D_sD_u=\delta D_u$, hence $\ld(w)=\ld(u)$. Moreover, since $D_{w^{-1}}=D_{u^{-1}}$, we have that $\ld(w^{-1})=\ld(u^{-1})$, so $\rd(w)=\rd(u)$. 

\end{proof} 

\begin{Corollary}
\label{corollary:archidiscese}
    Let $w\in \fc(\D)$. 
    \begin{itemize}
        \item[(a)] For $1\le i\le n+1$, the edge $\{i,i+1\}$ appears in $D_w$ if and only if $s_i\in \ld(w)$.
        \item[(b)] The edge $\{1,2\}_{\bullet}$ appears in $D_w$ or $\{1,2\}$ appears in $D_w$ and the numbers of $\bO$ in $D_w$ and $D_{0}D_{w}$ are the same, if and only if $s_0\in \ld(w)$. 
        \item[(c)] The edge $\{n+1,n+2\}_{\circ}$ appears in $D_w$ or $\{n+1,n+2\}$ appears in $D_w$ and the numbers of $\wO$ in $D_w$ and $D_{n+2}D_{w}$ are the same, if and only if $s_{n+2}\in \ld(w)$. 
    \end{itemize}
\end{Corollary}

\begin{proof}
The sufficient condition is proved in Lemma \ref{discesearchi}. Here we settle the other implication.

First, note that if $\{i,i+1\}$ appears in $D_w$ with $1\le i\le n+1$, then $D_iD_w=\delta D_w$, so $s_i\in \ld(w)$ by Theorem \ref{maindis}. Hence, (a) is proved. 

If $\{1,2\}_{\bullet}$ is present in $D_w$, then $D_0D_w=\delta D_w$, thus $s_0\in \ld(w)$ by Theorem \ref{maindis}. Suppose now that $\{1,2\}$ appears in $D_w$. Since the concatenation $D_0$ with $D_w$ forms a loop $\bO$ and the numbers of $\bO$ in $D_w$ and $D_{0}D_{w}$ are the same, it means that a $\bO$ giving rise to a reduction of type (r3) was already present in $D_w$. Hence $D_0D_w=\delta D_w$ and (b) follows by Theorem \ref{maindis}. Point (c) is analogous to (b).
\end{proof}

%%%%%%%%%%%%%%%%%%%%%%%%%%%%%%%%%%%%%%%%%%%
\section{Injectivity of $\widetilde{\theta}_D$}\label{sec:injectivity}
%%%%%%%%%%%%%%%%%%%%%%%%%%%%%%%%%%%%%%%%%%%
In this section we prove that $\widetilde{\theta}_D$ is injective. The argument proceeds in three steps:
we first show that $\widetilde{\theta}_D$ is injective on the set of irreducible elements of 
$\fc(\D).$
We then prove that if $\widetilde{\theta}_D(b_u)=\widetilde{\theta}_D(b_w)$ and one of $u$ and $w$ is irreducible, then $u=w$.
Finally, combining the previous steps we deduce the injectivity of $\widetilde{\theta}_D$.

\begin{Proposition}
\label{inietirr}
Let $u,w \in \I(\D)$. If $D_u=D_w$, then $u=w$.
\end{Proposition}

\begin{proof}
Since $D_u=D_w$, then $\ld(u)=\ld(w)$ and $\rd(u)=\rd(w)$ by Theorem $\ref{maindis}$. Suppose $u$ and $w$ are not in the same family. Note that one of them must be in the family $\TTC(\D)$ because the cardinalities of the descent sets of a candy and a complete zigzag are never equal. Hence assume $u\in \TTC(\D)$. If $w\in \zz(\D)$, necessarily $\ld(u)=\rd(u)=\ld(w)=\rd(w)$, hence \[\ld(w)=\rd(w)=\{s_0,s_1\}\mbox{ or }\ld(w)=\rd(w)=\{s_{n+1},s_{n+2}\}.\]
Therefore,  $3\le f_\bullet(w)+f_\circ(w)\ne f_\bullet(u)+f_\circ(u)=1$ which is a contradiction by Proposition $\ref{lemmadelta}$. If $w\in \car(\D)$ then we have a contradiction because $D_w$ has at least one $\bwO$ while $D_u$ has no occurrence of $\bwO$ by Proposition \ref{diagrammiasscaramelle} and Proposition \ref{diagrammiasscomm}(b).

Assume now that $u$ and $w$ are in the same family.
\begin{itemize}
    \item If $u,w\in \TTC(\D)$, then $u=w$ trivially.
    \item If $u,w\in \car(\D)$, then consider $u=u_0u_1\cdots u_m$ and $w=w_0w_1\cdots w_p$ their CFNF. Since $\ld(u)=\ld(w)$ and $D_w$ and $D_u$ have the same number of $\bwO$, then $u_0=w_0$ and $m=p$, by Proposition \ref{diagrammiasscaramelle}. Therefore, by Definition \ref{caramella}, $u=w$.
    \item If $u,w\in \zz(\D)$, then $\ld(u)=\ld(w)$ implies that $u$ and $w$ are of the same form (1) or (2) in Definition \ref{def:zigzag}. Moreover, the corresponding $k$ and $h$ are equal as well, since $f_\bullet(w)=f_\bullet(u)$ and $f_\circ(w)=f_\circ(u)$ by Proposition \ref{diagrammiasszigzag}.
\end{itemize}
\end{proof}

\begin{Lemma}
\label{lemmadiagcommuno}
Let $w\in \fc(\D)$ such that $f_\bullet(w)=f_\circ(w)=0$. If the edge $\{i,i'\}$ not decorated appears in $D_w$, then the following hold.
\begin{enumerate}
    \item[(a)] If $i=1$, then $s_0,s_1\notin supp(w)$.
    \item[(b)] If $i=2$, then $s_0,s_1,s_2\notin supp(w)$.
    \item[(c)] If $3\le i\le n$, then $s_{i-1},s_i\notin supp(w)$.
    \item[(d)] If $i=n+1$, then $s_{n},s_{n+1},s_{n+2}\notin supp(w)$.
    \item[(e)] If $i=n+2$, then $s_{n+1},s_{n+2}\notin supp(w)$.
\end{enumerate}
\end{Lemma}

\begin{proof}
We only prove point (c) since the other cases are similar. Then set $3\le i\le n$, and let $w=u_0\cdots u_m$ be its CFNF. We have $D_w=D_{u_0}\cdots D_{u_m}$ with $D_{u_{j}}$, $0\le j\le m$ described in Proposition $\ref{diagrammiasscomm}$, since $u_j\in \TTC(\D)$. First, note that $s_{i-1},s_i\notin supp(u_0)=\ld(w)$ by Lemma \ref{discesearchi}. Now we prove the result by contradiction. Let $1\le l \le m$ be the smallest index such that $s_{i-1}$ or $s_i$ is in $supp(u_l)$, and assume, without loss of generality, that $s_{i-1}\in supp(u_l)$. By Lemma \ref{discesearchi}, $\{i-1,i\}$ and $\{(i-1)',i'\}$ are in $D_{u_{l}}$, and by Lemma \ref{diagrammiasscomm} and Corollary \ref{corollary:archidiscese}, the undecorated edge $\{i,i'\}$ belongs to $D_{u_j}$, for all $0\le j<l$. Note that the configuration depicted in Figure \ref{fig:placeholder}(a) never appears by fully commutativity. Therefore, since $\{i,i'\}$ appears in $D_w$ by hypothesis, the only possibility is that there exist $h_1<h_2$ such that $s_0\in supp(u_{h_1})$ and $s_1\in supp(u_{h_2})$, or vice versa, and $s_0,s_1\notin supp(u_j)$ for all $h_1<j<h_2$, see Figure \ref{fig:placeholder}(b). Suppose $h_1$ is minimal, then in the diagram $D_{u_0}D_{u_1}\cdots D_{u_l}\cdots D_{u_{h_1}}\cdots D_{u_{h_2}}$ the edge $\{i,2'\}$ is decorated with a $\bullet$. Now, since $D_w$ contains $\{i,i'\}$, the edge $\{2, i'\}$ in $D_{u_{h_2+1}}\cdots D_{u_m}$ either is undecorated or its first decoration is a $\circ$. In any case, $\{i,i'\}$ is decorated since $f_\bullet(w)=f_\circ(w)=0$ by hypothesis, which is a contradiction. The same argument holds if we assume that $s_i\in supp(u_l)$, by exchanging the $\circ$ decorations with the $\bullet$.

\end{proof}

\begin{figure}
    \centering    \includegraphics[width=0.5\linewidth]{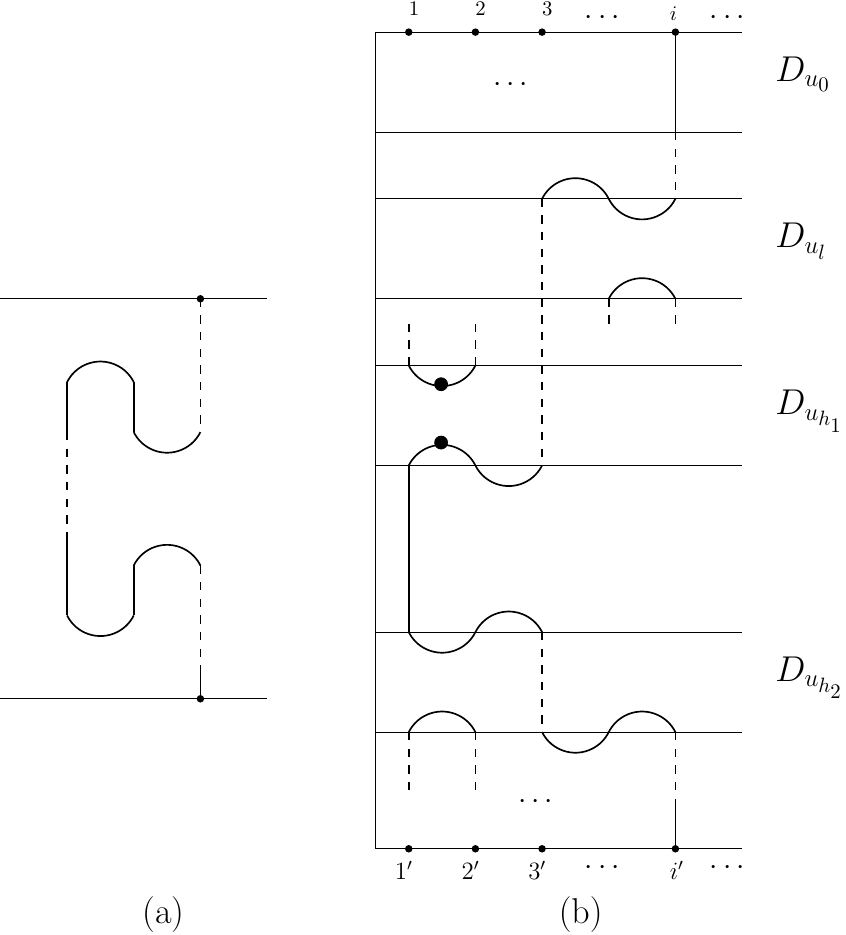}
    \caption{Concatenation $D_w=D_{u_0}\cdots D_{u_m}$ for the proof of Lemma \ref{lemmadiagcommuno}.}
    \label{fig:placeholder}
\end{figure}

\begin{Proposition}
\label{finalcomm}
Let $w\in \fc(\D)$. If $u\in \TTC(\D)$ such that $D_u=D_w$, then $u=w$.
\end{Proposition}

\begin{proof}
We distinguish four cases.
\begin{itemize}
\item[(1)] Assume $f_\bullet(w)=f_\circ(w)=0$. If $w\in \I(\D)$, then $u=w$ by Proposition \ref{inietirr}. Otherwise, we suppose that $w\red_{s_i} s_{i-1}w$, with $3\le i\le n-1$. Therefore since $s_{i-1}s_{i}$ is a prefix of $w$, we have that $s_i,s_{i+1}\notin \ld(w)$. Moreover, since $D_u=D_w$, the non-decorated edge $\{i+1,(i+1)'\}$ occurs in $D_w$ by Proposition \ref{diagrammiasscomm} and Corollary \ref{corollary:archidiscese}, which is not possible by Lemma $\ref{lemmadiagcommuno}$. The cases when $i\in \{2,n,n+1\}$, or $s_0s_2$, $s_{n}s_{n+2}$ or $s_{j}s_{j'}$, with $m_{s_j,s_{j'}}=3$ and $j>j'$, is a prefix of $w$ can be treated in the same way using Lemma \ref{lemmadiagcommuno}. So $w$ is necessarily irreducible and $u=w$. 
    \item[(2)] If $f_\bullet(w)=1$ and $f_\circ(w)=0$, then $s_0,s_1\in \ld(w)\cap \rd(w)$, because $\ld(u)=\rd(u)=\ld(w)=\rd(w)$ by Theorem $\ref{maindis}$. Moreover, it is possible to write $w=s_0s_1w'$ reduced and $s_0,s_1,s_2\notin supp(w')$. Let $u=s_0s_1u'$ reduced and $u'\in \TTC(\D)$. Observe that, $\{1,1'\}$ and $\{2,2'\}$ appear in $D_{w'}$ and in $D_{u'}$, since $s_0,s_1,s_2\notin supp(w')\cup supp(u')$. Then, $D_{u'}=D_{w'}$ since the portions of $D_{w'}$ and $D_{u'}$ on the right of $\{2,2'\}$ are both equal to the corresponding portion of $D_u=D_w$; for an example see Figure $\ref{aiutfig}$. Therefore, $u'=w'$ by point (1), thus $u=w$. 
    \item[(3)] If $f_\bullet(w)=0$ and $f_\circ(w)=1$, then this case is analogous to the previous one. 
    \item[(4)] If $f_\bullet(w)=f_\circ(w)=1$, then $w=s_0s_1w'$ and $u=s_0s_1u'$ both reduced, $s_0,s_1,s_2\notin supp(w')\cup  supp(u')$ and $D_{u'}=D_{w'}$ for the same argument in point (2). Moreover, $f_\bullet(w')=0$ and $f_\circ(w')=1$, so $u'=w'$ by point (3). Hence, $u=w$. 
\end{itemize}
\end{proof}

\begin{figure}[h!]
    \centering
    \includegraphics[width=0.4\linewidth]{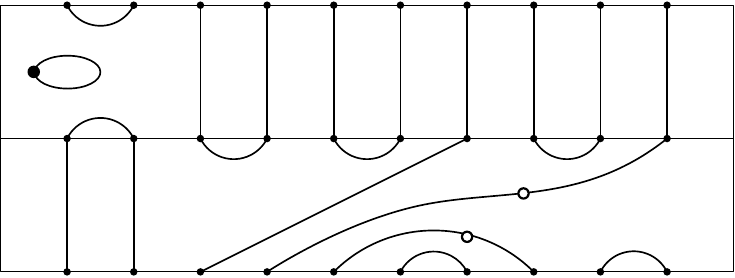}
    \caption{$D_w=D_{s_0s_1}D_w'$, $w=s_0s_1w'$ and $w'=s_3s_5s_8s_4s_6s_{10}s_5s_7s_8s_6s_9$.}
    \label{aiutfig}
\end{figure}

\begin{Proposition}
\label{proposition:zzecarcontr}
Let $w\in \fc(\D)$. If $u\in \zz(\D)\cup \car(\D)$ such that $D_u=D_w$, then $u=w$.
\end{Proposition}

\begin{proof}
If $u\in \zz(\D)$, then $\ld(w),\rd(w)\in \{\{s_0,s_1\},\{s_{n+1},s_{n+2}\}\}$ by Theorem $\ref{maindis}$. Hence, $w$ is both left and right irreducible by Lemma \ref{lemma:condzigzag}, so $w\in \I(\D)$. Then, $u=w$ by Proposition \ref{inietirr}. 

If $u\in \car(\D)$, by Theorem \ref{maindis} we can assume $\ld(u)=\ld(w)$ and $\rd(u)=\rd(w)$. Therefore, $w\in \I(\D)$ since it cannot admit any prefix or suffix of the form $st$, $m_{s,t}=3$. Then, $u=w$ by Proposition \ref{inietirr}.
\end{proof}

\begin{Theorem}
\label{lastresult}
The map $\widetilde{\theta}_D:\tl(\D) \rightarrow \Di(\D)$ is a $\Z[\delta]$-algebra isomorphism.
\end{Theorem}

\begin{proof}
First, we have that $\widetilde{\theta}_D$ is a surjective homomorphism by definition. So we only need to show that the set $\{D_w\in \Di(\D)\mid w\in \fc(\D)\}$ is a basis for $\Di(\D)$. Since the diagram $D_w$, for all $w\in \fc(\D)$, does not contain any undecorated loops by Proposition \ref{lemmadelta}, it is sufficient proving that if $u,w \in \fc(\D)$ such that $D_u=D_w$ then $u=w$. Suppose that $w$ is reducible to $v\in\I(\D)$ by a sequence $w_0=w\red w_1\red w_2\red \cdots \red w_l=v$.
We proceed by induction on $l$. If $l=0$, then $w\in \I(\D)$, hence $u=w$ by Propositions $\ref{finalcomm}$ and $\ref{proposition:zzecarcontr}$. Without loss of generality, suppose that $w\red_t sw=:w_1$. We have that $s\in \ld(u)$, since $\ld(u)=\ld(w)$ by Theorem \ref{maindis}. Suppose that $st$ is not a prefix of $u$, then $tu\in \fc(\D)$ reduced by Remark \ref{oss}(5). Moreover, $D_{w_1}=D_tD_w=D_{t}D_u=D_{tu}$ so by inductive hypothesis, it means $w_1=tu$. But this is a contradiction since $ts$ is not a prefix of $w_1$ but it is a prefix of $tu$. Therefore, $u\red_t su=:u_1$. Hence, $D_{w_1}=D_tD_w=D_tD_u=D_{u_1}$ so by inductive hypothesis, $u_1=w_1$ and $u=w$.  
\end{proof}
%%%%%%%%%%%%%%%%%%%%%%%%%%%%%%%%%%%%%%%%%%%
\section{Some remarks on the Lusztig's $\ab$-function}\label{sec:afunc}
%%%%%%%%%%%%%%%%%%%%%%%%%%%%%%%%%%%%%%%%%%%
In \cite{lusztig_cells}, Lusztig introduced the function $\ab:W\rightarrow\mathbb{N}_0$ which plays an important role in Kazdhan--Lusztig theory. For instance, this function is constant on the $LR$-cells of a Coxeter group \cite[Theorem 5.4]{lusztig_cells}. In \cite{shi_fully}, Shi provided a more manageable way to compute the $\ab$-function of a $\fc$ element of a finite or affine Weyl group $W$ through its heap. Let $w\in \fc(W)$, define $n(w)$ as the cardinality of the maximum antichain in the heap $H(w)$. Equivalently, by \cite[Lemma 2.7]{shi_fully}, $n(w)$ is the maximum integer $k$ such that $w=uw'v$ is reduced, $\ell(w')=k$ and $w'\in\TTC(W)$.
In \cite[Theorem 3.1]{shi_fully}, Shi proved the following important result.
\begin{Theorem}
    Let $w\in \fc(W)$. Then $\ab(w)=n(w)$.
\end{Theorem}

In \cite[Lemma 5.4, 5.5]{ErnstDiagramII}, Ernst proved that when $W$ is an affine Coxeter system of type $\widetilde{C}_n$, then $n(w)=\ab(D_w)$, where $D_w$ is the decorated diagram associated to $w\in \fc(\widetilde{C}_n)$ in the diagrammatic realization of $\tl(\widetilde{C}_n)$ he introduced in \cite{ErnstDiagramI,ErnstDiagramII}, and $\ab(D_w)$ is the number of non-propagating edges on the north face of $D_w$. However, this property does not hold for the diagrammatic realization of $\tl(\D)$ used in this paper. For instance, if $w=s_0s_1$, we have $2=n(s_0s_1)\neq \ab(D_w)=1$. Our goal is to provide a function defined on the $LR$-decorated diagrams of type $\D$, that coincides with Lusztig's $\ab$-function.

\begin{Definition}

Let $w \in \fc(\D)$. We define the function $\at(D_w)$ as follows:
$$
\at(D_w) :=
\begin{cases}
\ab(D_w) + \#\bO + \#\wO, & \text{if } \ab(D_w) > 1; \\
\ab(D_w) + \lambda, & \text{if } \ab(D_w) = 1,
\end{cases}
$$
where
$$
\lambda :=
\begin{cases}
0, & \text{if } \#\bO = \#\wO = 0; \\
1, & \text{otherwise}.
\end{cases}
$$

\end{Definition}

\begin{Theorem}\label{theorem:a-lusztig}
    If $w\in \fc(\D)$, then $n(w)=\at(D_w)$.
\end{Theorem}

\begin{proof}
    By the definitions of irreducible FC elements it is immediate to compute that $n(w)=|\ld(w)|$ for $w\in \I(\D)$, that is:
 $$n(w)=\begin{cases}
        \ell(w), &\mbox{if }w\in \TTC(\D);\\
        2, &\mbox{if }w\in \zz(\D);\\
        n/2 +1, &\mbox{if }w\in \car(\D).
    \end{cases}$$
     By Proposition \ref{diagrammiasszigzag}, if $w\in \zz(\D)$ then $\ab(D_w)=1$. By Lemma \ref{discesearchi}, if $w\in \car(\D)$, then $D_w$ has $n/2+1$ non-propagating edges. If $w\in \TTC(\D)$, by Corollary \ref{corollary:archidiscese}, $\ab(D_w)=\ell(w)-f_\bullet(w)-f_\circ(w)$. Hence, if $w\in \I(\D)$, $n(w)=\at(D_w)$ by Proposition \ref{lemmadelta}. Moreover, by \cite[Lemma 2.9]{shi_fully}, $n(w)$ is invariant by star reducibility. On the other hand, by Remarks \ref{rem:occirr}(b) and \ref{remark:a-value}, and by Proposition \ref{lemmadelta}, $\ab(D_w)$, $\# \bO$ and $\# \wO$ are invariant by star reducibility as well. 
     Hence $n(w)=\at(D_w)$ for any $w\in \fc(\D)$. 
\end{proof}
%%%%%%%%%%%%%%%%%%%%%%%%%%%%%%%%%%%%%%%%%%
\bibliographystyle{plain}
\bibliography{biblio}

\end{document}